\documentclass[reqno]{amsart}
\usepackage{amssymb,amscd,verbatim, amsthm,graphics, color,latexsym,amsmath,multicol}
\usepackage{fancybox}
\usepackage{longtable}

\usepackage[all]{xy}

\newcommand \fk[1]{{{\mathfrak #1}}}
\newcommand \C[1]{{\mathcal #1}}

\newcommand \wti[1]{{\widetilde {#1}}}

\newcommand\fg{\mathfrak g}

\newcommand \bC{{\mathbb C}}

\newcommand \bH{{\mathbb H}}
\newcommand \bR{{\mathbb R}}
\newcommand \bZ{{\mathbb Z}}

\newcommand\CO{{\C O}}

\newcommand\ep{{\epsilon}}

\newcommand\al{{\alpha}}

\newcommand\fa{{\mathfrak a}}
\newcommand\fh{{\mathfrak h}}

\newtheorem{theorem}{Theorem}[section]

\newtheorem{corollary}[theorem]{Corollary}
\newtheorem{lemma}[theorem]{Lemma}
\newtheorem{proposition}[theorem]{Proposition}
\newtheorem{definition}[theorem]{Definition}
\newtheorem{remark}[theorem]{Remark}
\newtheorem{example}[theorem]{Example}

\newcommand\Hom{\operatorname{Hom}}

\newcommand\tr{\operatorname{tr}}

\newcommand\im{\operatorname{im}}

\newcommand\triv{\mathsf{triv}}

\newcommand\Irr{\mathsf{Irr}}

\newcommand\gr{\mathsf{gr}}
\newcommand\Pin{\mathsf{Pin}}
\newcommand\ab{\mathsf{ab}}

\newcommand\bh{\mathbf {h}}
\newcommand\bfH{\mathbf {H}}
\newcommand\bfA{\mathbf {A}}
\newcommand\ba{\mathbf {a}}

\newcommand\bom{\mathbf{\omega}}
\newcommand\Id{\operatorname{Id}}
\newcommand\Spec{\operatorname{Spec}}

\def\<{\langle} 
\def\>{\rangle}

\numberwithin{equation}{subsection}

\begin{document}

\title{Dirac cohomology for symplectic reflection algebras}

\author{Dan Ciubotaru}
        \address[D. Ciubotaru]{Mathematical Institute\\ University of
          Oxford\\ Oxford, OX2 6GG, UK}
        \email{dan.ciubotaru@maths.ox.ac.uk}

\thanks {It is a pleasure to thank B. Kr\"otz and E. Opdam for the invitation to give a series of lectures on the theory of the Dirac operator for Hecke algebras at the Spring School ``Representation theory and geometry of reductive groups'', Heiligkreuztal 2014, where some of these ideas  crystallized. I also thank J.S. Huang, K.D. Wong, and the referee for corrections, helpful comments, and references.}

\begin{abstract} We define uniformly the notions of Dirac operators and Dirac cohomology  in the framework of the Hecke algebras introduced by Drinfeld \cite{Dr}. We generalize in this way, the Dirac cohomology theory for Lusztig's graded affine Hecke algebras defined in \cite{BCT} and further developed in \cite{BCT,COT,Ci,CH,Cha2}. We  apply these constructions to the case of the symplectic reflection algebras defined by Etingof-Ginzburg \cite{EG}, particularly to rational Cherednik algebras for real or complex reflection groups. As applications, we give criteria for unitarity of modules in category $\CO$ and we show that the $0$-fiber of the Calogero-Moser space admits a description in terms of a certain ``Dirac morphism'' originally defined by Vogan for representations of real reductive groups. 
\end{abstract}

\maketitle

\setcounter{tocdepth}{1}
\tableofcontents

\section{Introduction}
\subsection{}The Dirac operator has played an important role in the representation theory of real reductive groups, see for example \cite{AS}, \cite{Ko}, \cite{Pa}, and the monograph \cite{HP2}. The notion of Dirac cohomology for  admissible $(\fg,K)$-modules of real reductive groups was introduced by Vogan \cite{Vo} around 1997. The Dirac cohomology of a $(\fg,K)$-module  is a certain finite dimensional representation of (a pin cover of) the maximal compact subgroup $K$. One of the main ideas, ``Vogan's conjecture'', proved by Huang and Pad\v zi\'c \cite{HP} says that, if nonzero, the Dirac cohomology of an irreducible module $X$ uniquely determines the infinitesimal character of $X$. 

\subsection{}Motivated by the analogy between the theory of graded affine Hecke algebras $\bH$ of reductive $p$-adic groups, as defined by Lusztig \cite{Lu}, and certain elements of the representation theory of real reductive groups, in joint work with Barbasch and Trapa \cite{BCT}, we defined a Dirac operator and the notion of Dirac cohomology for $\bH$-modules. This theory was subsequently developed in several papers, including \cite{COT,CT,CH,Cha2} and it led to interesting results, such as a geometric realization in the kernel of global Dirac operators of the irreducible discrete series $\bH$-modules \cite{COT}, a partial analogue of the realization of discrete series representations for real semisimple groups by Atiyah-Schmid \cite{AS} and Parthasarathy \cite{Pa}. An important element that occurs in these constructions for $\bH$ is a certain pin cover $\wti W$ of finite Coxeter groups $W$ whose representations turned out to have surprising relations with the geometry of the nilpotent cone, see \cite{Ci,CH,CT}, also \cite{Cha} for noncrystallographic Coxeter groups.  In particular, when $\bH$ has equal parameters, the Dirac morphism stemming from the analogue of Vogan's conjecture for $\bH$ leads to a map 
\begin{equation}
\zeta^*: \Irr(\wti W)\to \Spec (Z(\bH))=\fh^*/W,
\end{equation}
whose image lies in the set of $W$-conjugates of neutral elements for Lie triples corresponding to nilpotent orbits whose connected centralizers are solvable. (Here, one regards $\fh^*$ as a Cartan subalgebra of the Langlands dual complex Lie algebra.)

\subsection{}The graded affine Hecke algebra is a particular case of a more general class of algebras $\bfH_\ba$ (Definition \ref{d:Drinfeld}) introduced by Drinfeld \cite{Dr}, see also Ram-Shepler \cite{RS}. These associative unital $\bC$-algebras are defined in terms of a finite group $W$ that acts linearly on a finite dimensional complex vector space $V$, and a family of skew-symmetric forms $\ba=(a_w)_{w\in W}$. The requirement is that $\bfH_\ba$ satisfies a Poincar\'e-Birkhoff-Witt type property. In the present paper, we define uniformly the notions of Dirac operator $\C D$ and Dirac cohomology  in the general framework of algebras $\bfH_\ba$. The only additional hypothesis that we need is that $V$ has a non degenerate $W$-invariant symmetric bilinear form. We prove the basic facts about $\C D$ such as:
\begin{itemize}
 \item the formula for $\C D^2$ (Theorem \ref{t:D2}), and 
 \item the analogous version of Vogan's conjecture for Dirac cohomology (Theorem \ref{t:ker-dW} and Theorem \ref{t:vogan-conj}), 
 \end{itemize}
 generalizing in this way the results of \cite{BCT}. We  apply these constructions to the case of symplectic reflection algebras defined by Etingof and Ginzburg \cite{EG}, in particular, to rational Cherednik algebras $\bfH_{t,c}$ of real and complex reflection groups and arbitrary parameters $t,c$. (A different Dirac operator and the formula for its square in the setting of rational Cherednik algebras of finite Weyl groups at $t=0$ was previously defined in unpublished joint work with Trapa \cite{CT2}.)

The analogy for the Dirac theories is as follows: one should think that the Dirac operator for the graded affine Hecke algebra $\bH$ corresponds to that in the symmetric case of a real reductive group (in fact, this is more than an analogy), while the one for the rational Cherednik algebras $\bfH_{t,c}$ corresponds to the Kostant's cubic Dirac operator \cite{Ko} for the minimal parabolic case, i.e., the Verma modules case. 

\subsection{}
We give two applications of these methods in the paper. 

Firstly, we establish criteria for unitarity of modules in the category $\CO$ introduced by Ginzburg-Guay-Opdam-Rouquier \cite{GGOR}, in particular strengthening a non-unitarity criterion from \cite{ES}. 

Secondly, as proved in \cite{EG}, $\Spec(Z(\bfH_{0,c}))$ can be thought of as a generalization of the Calogero-Moser space studied in \cite{KKS}. We show that the $0$-fiber $\Upsilon^{-1}(0)$ defined by Gordon \cite{Go} in the generalized Calogero-Moser space  admits a description in terms of a certain ``Dirac morphism'' (see (\ref{e:spec-0}), Theorem \ref{t:CM} and Corollary \ref{c:zeta=theta}) that appears in the formulation of the analogue of Vogan's conjecture:
\begin{equation}
\zeta^*_c:\Irr(W)\to \Upsilon^{-1}(0).
\end{equation} 
It is expected, see \cite{GM}, that the Calogero-Moser partition of $\Irr(W)$ according to $\Upsilon^{-1}(0)$  refines (and often coincides) with the Lusztig-Rouquier partition  \cite{Lu2,Ro} of $\Irr(W)$ into families. The identification of the Calogero-Moser partition with the ``Dirac partition" given by the morphism $\zeta^*_c$ opens up the perspective of studying these conjectures from the point of view of the Dirac operator. For first results in this direction, see \cite{Ci2}.

\section{The Dirac operator for Drinfeld's Hecke algebras}

\subsection{Drinfeld's graded Hecke algebra} Let $V$ be a finite dimensional complex vector space  and $W$  a finite subgroup of $GL(V)$. Suppose that for every $w\in W$, we have a skew-symmetric bilinear form
\begin{equation}
a_w: V\times V\to \bC.
\end{equation}
Denote the family of skew-symmetric forms by $\ba=(a_w)_{w\in W}.$

\begin{definition}[\cite{Dr}]\label{d:Drinfeld} Let $T(V)$ be the tensor algebra of $V$. Define the algebra $\bfH=\bfH_\ba$ to be the quotient of $T(V)\otimes \bC[W]$ by the relations:
\begin{equation}
\begin{aligned}
&w\cdot v\cdot w^{-1}=w(v),\\
&[u,v]=\sum_{w\in W} a_w(u,v) w,\\
\end{aligned}
\end{equation}
for all $u,v\in V$ and $w\in W.$ Define a filtration of $\bfH_\ba$ by giving the elements in $\bC[W]$ degree $0$ and the elements $v\in V$ degree $1$. The algebra $\bfH_\ba$ is called a Drinfeld graded Hecke algebra if it has the PBW property, i.e., the associated graded object is isomorphic to $\bfH_0=S(V)\rtimes \bC[W].$
\end{definition}

The PBW property imposes strict conditions on the forms $a_w$. Define 
\begin{equation}
W(\ba)=\{w\in W: a_w\neq 0\}.
\end{equation}
Let $\ker a_w$ be the radical of the form $a_w.$ Let $V^w$ be the subspace of fixed points elements of $w$ in $V$. 

\begin{proposition}[\cite{Dr},{\cite[Theorem 1.9]{RS}}]\label{p:ram} The algebra $\bfH$ has the PBW property if and only if the following properties hold simultaneously: 
\begin{enumerate}
\item For every $h\in W$, $a_{h^{-1}wh}(u,v)=a_w(h(u),h(v)),$ for all $u,v\in V$.
\item For every $w\in W(\ba)\setminus\{1\}$,  $\ker a_w=V^w$ and $\dim V^w=\dim V-2$.
\item For every $w\in W(\ba)\setminus\{1\}$ and every $h\in Z_W(w)$, $\det(h|_{(V^w)^\perp})=1,$ where $(V^w)^\perp=\{v-w(v):v\in V\}$.
\end{enumerate}
\end{proposition}
In particular, this means that $W(\ba)$ is a union of conjugacy classes in $W$.
We assume from now on that $\bfH$ has the PBW property.

\subsection{The Clifford algebra} Because of the constructions in Clifford theory, we need to assume that $V$ carries a $W$-invariant non degenerate symmetric bilinear form $\langle~,~\rangle$. Let $O(V)$ be the corresponding orthogonal group. Then we have  $W\subset O(V).$ Let $\det$ denote the determinant character of $O(V)$. For every subspace $U$ of $V$, denote by $U^\perp$ the orthogonal complement of $U$ in $V$ with respect to $\<~,~\>.$ This notation is compatible with the definition of $(V^w)^\perp$ from Proposition \ref{p:ram}.

\smallskip

Let $C(V)$ be the complex Clifford algebra defined by $V$ and $\<~,~\>$. For a survey of Clifford algebra theory see \cite{HP2} and \cite{Me}. 
The defining relation of $C(V)$ is
\begin{equation}
v\cdot v'+v'\cdot v=-2\langle v,v'\rangle,\text{ for all }v,v'\in V.
\end{equation}
The algebra $C(V)$ has a filtration $(C^n(V))$ by degrees and the associated graded object is $\bigwedge V.$ It also has 
 a $\bZ/2\bZ$-grading given by the parity of degrees. Denote the $\bZ/2\bZ$-grading of $C(V)$ by $C(V)=C(V)_0\oplus C(V)_1.$ Define an algebra automorphism  
 \begin{equation}
 \ep:C(V)\to C(V), \text{ such that } \ep=\Id\text{ on }C(V)_0 \text{ and }\ep=-\Id\text{ on }C(V)_1,
 \end{equation}
We also extend $\ep$ to an automorphism of $\bfH\otimes C(V)$ by making it the identity on $\bfH.$

Define the transpose of $C(V)$ as the anti-involution:
 \begin{equation}
 { }^t:C(V)\to C(V),\ v^t=-v,\ v\in V,\ (ab)^t=b^t a^t,\ a,b\in C(V).
\end{equation}
The Pin group is the subgroup of the units of $C(V)$ defined by
\begin{equation}
\Pin(V)=\{a\in C(V)^\times: \ep(a)\cdot V\cdot a^{-1}\subset V,\ a^t=a^{-1}\}.
\end{equation}
This is a central double cover of $O(V)$ where the projection is given by $p:\Pin(V)\to O(V)$, $p(a)(v)=\ep(a)\cdot v\cdot a^{-1},$ $a\in\Pin(V)$, $v\in V$. 

Suppose $v\in V$ is such that $\<v,v\>\neq 0.$ Let $s_v\in O(V)$ be the reflection with respect to the hyperplane perpendicular on $v$, i.e., $s_v(x)=x-\frac 2{\<v,v\>}\<x,v\>v.$ Then
\begin{equation}\label{e:preimage}
p^{-1}(s_v)=\left\{\pm \frac 1{|v|} v\right\}\in \Pin(V),
\end{equation}
where $|v|$ is a choice of $\sqrt{\<v,v\>}$. 

Define the pin double cover of $W$ as
\begin{equation}
\wti W:=p^{-1}(W)\subset \Pin(V).
\end{equation}

\begin{lemma}\label{l:det}
For every $a\in \Pin(V)$, we have $\ep(a)a^{-1}=\det(p(a))$, where $\det$ is the determinant character of $O(V)$.
\end{lemma}

\begin{proof}
It is sufficient to verify the claim for $a=\frac 1{|v|}v\in V$ where $|v|\neq 0$. Then $\ep(a)a^{-1}=-a\cdot a^{-1}=-1$. On the other hand $p(a)=s_v$, and $\det(s_v)=-1.$
\end{proof}

\subsection{The Dirac element} Let $\{v_i\}$ and $\{v^i\}$ be dual bases of $V$ with respect to the form $\<~,\>.$ Define the Dirac element
\begin{equation}
\C D=\sum_i v_i\otimes v^i\in\bfH\otimes C(V).
\end{equation}
This definition does not depend on the choice of dual bases.

Define a group homomorphism $\Delta: \wti W\to \bfH\otimes C(V)$ by setting
\begin{equation}
\Delta(\wti w)=p(\wti w)\otimes \wti w,\quad \wti w\in\wti W.
\end{equation}
Extend this linearly to a map $\Delta:\bC[\wti W]\to\bfH\otimes C(V).$

\begin{lemma}\label{l:D-inv}
For every $\wti w\in \wti W$, we have the invariance property:
\begin{equation}
\Delta(\wti w)\C D\Delta(\wti w^{-1})=\det(p(\wti w)) \C D.
\end{equation}
\end{lemma}

\begin{proof}
We have $\Delta(\wti w)\C D\Delta(\wti w^{-1})=\sum_i p(\wti w) v_i p(\wti w)^{-1}\otimes \wti w\cdot v^i\cdot \wti w^{-1}=\sum_i p(\wti w)(v_i)\otimes (\wti w \ep(\wti w)^{-1}) p(\wti w)(v^i).$ The claim follows from Lemma \ref{l:det} and the fact that $\C D$ is independent of the bases.
\end{proof}

One of the main tools is the computation of $\C D^2$. Define
\begin{equation}\label{e:bh}
\bh=\sum_{i} v_i v^i\in \bfH.
\end{equation}
This element is independent of the dual bases and moreover, $\bh\in\bfH^W.$

For every $w\in W(\ba)$, define
\begin{equation}\label{e:kappa}
\kappa_w=\sum_{i,j} a_w(v_i,v^j)v^i v_j\in C(V).
\end{equation}
This is also independent of the bases, in fact it can be defined as follows. Since $a_w$ is a skew-symmetric bilinear form on $V$, $\kappa_w$ is the element of $C(V)$ obtained via the identifications:
\[a_w\in ({\bigwedge}^2 V)^*\cong {\bigwedge}^2(V^*)\cong {\bigwedge}^2 V\rightarrow C^2(V)\subset C(V).\]
The last inclusion comes from the Chevalley map $\bigwedge V\to C(V)$, which gives a section of the natural map $C(V)\to \bigwedge V$ given by taking the associated graded algebra. The identification $V^*\cong V$ is via our non degenerate form $\<~,~\>.$ 

If $w=1\in W(\ba)$, then by Proposition \ref{p:ram}, $a_1\in (({\bigwedge}^2 V)^*)^W$, and so $\kappa_1\in C^2(V)^W.$

We calculate $\C D^2$:
\begin{lemma} \label{l:D2}
\begin{equation}
\C D^2=-\bh\otimes 1+ \frac 12\sum_{w\in W(\ba)} w\otimes \kappa_w.
\end{equation}
\end{lemma}

\begin{proof} We compute directly:
\begin{align*}
2\C D^2&=\sum_{i,j}(v_iv^j\otimes v^i v_j+v^jv_i\otimes v_j v^i)\\
&=\sum_{i,j} (v_i v^j\otimes (v^i v_j+\<v^i,v_j\>)+v^j v_i\otimes (v_j v^i+\<v^i,v_j\>))-\sum_{i,j}\<v^i,v_j\>(v_i v^j+v^j v_i)\otimes 1.
\end{align*}
Now use that $\<v^i,v_j\>=\delta_{i,j}$ and that $v_j v^i+\<v^i,v_j\>=-(v^i v_j+\<v^i,v_j\>)$, and continue:
\begin{align*}
2\C D^2&=\sum_{i,j} [v_i,v^j]\otimes (v^i v_j+\delta_{i,j})-\sum_i (v_i v^i+v^i v_i)\\
&=\sum_{i,j}\sum_{w\in W(\ba)} a_w(v_i,v^j) w\otimes v^i v_j+\sum_i \sum_{w\in W(\ba)} a_w(v_i,v^i) w\otimes 1- 2\bh\otimes 1\\
&=-2\bh\otimes 1+\sum_{w\in W(\ba)} w\otimes \kappa_w,
\end{align*}
where we used that $\sum_{i} a_w(v_i, v^i)=\tr a_w=0$, since $a_w$ is skew-symmetric.
\end{proof}

We wish to relate $\kappa_w$ with $p^{-1}(w)\in \wti W$ for $w\in W(\ba)$. Let $w\in W(\ba)\setminus\{1\}$ be given. Then, Proposition \ref{p:ram} implies that the radical of $a_w$ equals $V^w$, and this is a codimension  $2$ subspace of $V.$ We first remark that
\begin{equation}
V=V^w\oplus (V^w)^\perp,
\end{equation}
is an orthogonal decomposition with respect to $\<~,~\>$. Indeed,  $w|_{{V^w}^\perp}$ has determinant $1$ by Proposition \ref{p:ram}, and therefore if one eigenvalue on $(V^w)^\perp$ is $1$, both are, which would give $(V^w)^\perp=0$, a contradiction.  So $V^w\cap (V^w)^\perp=0.$
Since $\<~,~\>$ is non degenerate on $V$, the restriction to $(V^w)^\perp$ must be non degenerate too. This means that if $\{v_i\}$, $\{v^i\}$, $i=1,2$, are dual bases of $(V^w)^\perp$, then 
\[\kappa_w=\sum_{i,j=1}^2a_w(v_i,v^j)v^iv_j,\]
so the calculation of $\kappa_w$ reduces to the two-dimensional case  $SO((V^w)^\perp).$

\begin{lemma}
The element $w\in W(\ba)\setminus \{1\}$ can be written as a product of two reflections $w=s_\al s_\beta$, where $\al,\beta\in (V^w)^\perp$ are linearly independent and $\<\al,\al\>=\<\beta,\beta\>=1$.
\end{lemma}

\begin{proof} Since $\<~,~\>$ is non degenerate on $(V^w)^\perp$, there exists $v\in (V^w)^\perp$ such that $\<v,v\>\neq 0$. If $\<wv-v,wv-v\>=0$, it follows by $W$-invariance that $\<wv-v,v\>=0$. Since the form is non degenerate, $wv-v$ and $v$ are not linearly independent, so $v$ is an eigenvector of $w$ with eigenvalue $\lambda$. But then $0=\<wv-v,wv-v\>=(\lambda-1)^2\<v,v\>$, so $\lambda=1$, contradiction.
Set $\al$ to equal a norm $1$ vector in the direction of $wv-v.$ Then $wv=s_\al(v)$, so $s_\al w$ fixes $\bC v.$ To finish let $\beta$ be a norm $1$ vector orthogonal to $v$ (which exists because the form is non degenerate).
\end{proof}

To calculate $\kappa_w$ in terms of $\al$ and $\beta$, choose $u_1=\al$, $u_2=\frac1{\sqrt{1-\<\al,\beta\>^2}}(\beta-\<\al,\beta\>\al)$, which is an orthonormal basis for $(V^w)^\perp.$  (Note that $\<\al,\beta\>\neq 1.$) Then
\begin{equation}
\begin{aligned}
\kappa_w&=2 a_w(u_1,u_2)u_1\cdot u_2=2\frac {a_w(\al,\beta)}{1-\<\al,\beta\>^2} (\al\beta+\<\al,\beta\>).
\end{aligned}
\end{equation}
From (\ref{e:preimage}), $p^{-1}(w)=\{\al\beta,-\al\beta\}\subset \wti W.$  Notice that if $\gamma=s_\al(\beta)$, then
\[w=s_\al s_\beta=s_\gamma s_\al,\text{ but }\al\beta=-\gamma\al\in C(V),\]
and thus what we call $\al\beta\in C(V)$ is only determined up to a sign. Nevertheless, denote
\begin{equation}\label{e:w-tilde}
\wti w=\al\beta,\quad c_{\wti w}=\frac {a_w(\al,\beta)}{1-\<\al,\beta\>^2}, \quad e_w=\frac {a_w(\al,\beta)\<\al,\beta\>}{1-\<\al,\beta\>^2}.
\end{equation}
Then $c_{\wti w}\wti w\in C(V)$ and $e_w\in \bC$ do not depend on the choices of unit vectors $\al,\beta$.

\smallskip

Thus, combining with  Lemma \ref{l:D2} we proved the following formula for $\C D^2$.

\begin{theorem}\label{t:D2}
The square of the Dirac element equals (in $\bfH\otimes C(V)$):
\begin{equation}
\C D^2=-\Omega_\bfH\otimes 1+ 1\otimes \frac 12\kappa_1+\Delta(\Omega_{\wti W,\ba}),
\end{equation}
where 
\begin{align}\label{e:Omega}
&\Omega_{\bfH}=\bh-\sum_{w\in W(\ba)\setminus\{1\}} e_w w\in \bfH^W,\\\label{e:Omega-W}
&\Omega_{\wti W,\ba}=\sum_{w\in W(\ba)\setminus\{1\}} c_{\wti w}\wti w\in \bC[\wti W]^{\wti W},\\\label{e:kappa1}
&\kappa_1=\text{ image in } C^2(V)^W\text{ of }a_1\in (({\bigwedge}^2V)^*)^W.
\end{align}
The elements $\bh$ are defined in (\ref{e:bh}), $\wti w, c_{\wti w}, e_w$ are defined in (\ref{e:preimage}). 
\end{theorem}

By analogy with the case of real semisimple Lie algebras \cite{Pa} or with the case of graded affine Hecke algebras \cite{BCT}, we may refer to $\Omega_{\bfH}$ and $\Omega_{\wti W,\ba}$ as Casimir elements of $\bfH$ and $\bC[\wti W]$, respectively.

\subsection{The element $\Omega_{\bfH}$} We analyze the element $\Omega_{\bfH}$.
Define the linear map $\mathbf j: V\to V$ by
\begin{equation}\label{e:i-map}
\mathbf j(x)=\sum_i a_1(x,v_i) v^i,\quad x\in V.
\end{equation}

\begin{proposition}\label{p:Omega-central}
The element $\Omega_{\bfH}$ defined in (\ref{e:Omega}) satisfies:
\begin{equation}
[x,\Omega_\bfH]=2 \mathbf j(x),\text{ for all }x\in V.
\end{equation}
In particular, $\Omega_{\bfH}$ is central in $\bfH$ if and only if $1\notin W(\ba)$, i.e., $a_1=0$.
\end{proposition}

\begin{proof}
Let $x\in\bfH$ be arbitrary. On the one hand 
\[[x,\bh]=\sum_{w\in W(\ba)}(\sum_i (a_w(x,v_i)w(v^i)+a_w(x,v^i)v_i)) w,
\]
and on the other
\[ [x,\sum_{w\in W(\ba)\setminus\{1\}}e_w w]=\sum_{w\in W(\ba)\setminus 1}e_w (x-w(x)) w.\]
Comparing the two formulas, we see that $x$ the claim comes down to the verification that for every $w\in W(\ba)\setminus\{1\}$,
\[\sum_i (a_w(x,v_i)w(v^i)+a_w(x,v^i)v_i)=e_w (x-w(x)).\]
If $x\in V^w$, then the right hand side is $0$, and so is the left hand side because $\ker a_w=V^w.$ 

Assume $x\in (V^w)^\perp$. The left hand side comes down to a calculation in $(V^w)^\perp.$ Suppose $v_1,v_2$ is an orthonormal basis of $(V^w)^\perp$, then for $x=v_1$, the relation to check is
\[a_w(v_1,v_2)(v_2+w(v_2))=e_w(v_1-w(v_1)).\]
(Similarly for $x=v_2$.)
Let $w=s_\al s_\beta$ as before, and choose an orthonormal basis of $(V^w)^\perp$ given by $v_1=\al$ and $v_2=\frac 1{\sqrt{1-\<\al,\beta\>^2}}( \beta-\<\al,\beta\>\al).$ Since the definitions are independent of the choice of $\al,\beta$, it is sufficient to check the relation for $x=v_1$. The left hand side becomes
\begin{align*}
a_w(v_1,v_2)(v_2+w(v_2))&=\frac {a_w(\al,\beta)}{1-\<\al,\beta\>^2}(\beta-\<\al,\beta\>\al+s_\al s_\beta(\beta)-\<\al,\beta\> s_\al s_\beta(\al))\\
&=\frac {a_w(\al,\beta)}{1-\<\al,\beta\>^2}(\beta-\<\al,\beta\>\al-\beta+2\<\al,\beta\>\al-\<\al,\beta\> s_\al s_\beta(\al))\\
&=\frac {a_w(\al,\beta)\<\al,\beta\>}{1-\<\al,\beta\>^2}(\al-s_\al s_\beta(\al))=e_w(v_1-w(v_1)).
\end{align*}
\end{proof}

We also notice the following commutation relation in $C(V)$.
\begin{lemma}\label{l:kappa-grading}
For every $x\in V$, $[x,\frac 12\kappa_1]=-2 \mathbf j(x)$ in $C(V)$, where $\mathbf j$ is defined in (\ref{e:i-map}).
\end{lemma}
\begin{proof}
We calculate
\begin{align*}
[x,\kappa_1]&=\sum_{i,j}a_1(v_i,v^j)([x,v^i]v_j+v^i[x,v_j])\\
&=2\sum_{i,j}a_1(v_i,v^j)(-v^ixv_j-\<x,v^i\>v_j+v^ixv_j+\<x,v_j\>v^i)\\
&=-2\sum_{i,j}a_1(v_i,v^j)\<x,v^i\>v_j+2\sum_{i,j}a_1(v_i,v^j)\<x,v_j\>v^i\\
&=-4\sum_{i,j}a_1(v_i,v^j)\<x,v^i\>v_j,\text{ because of the skew-symmetry of }a_w\\
&=-4\sum_j a_1(\sum_i \<x,v^i\>v_i,v^j) v_j=-4\sum_j a_1(x,v^j)v_j.
\end{align*}
\end{proof}

\subsection{Example} The motivating example for this generalization is the graded affine Hecke algebra $\bH$ introduced by Lusztig \cite{Lu}. Let $W$ be a finite Coxeter group acting on a Euclidean space $V_0$ and take $V=V_0\otimes_\bR \bC$. Let $\Phi\subset V_0^*$ denote a root system with $\Phi^+$ a choice of positive roots and let $k:\Phi\to\bC$ be a $W$-invariant function. The skew symmetric forms are given by the formula (see \cite{RS}):
\begin{equation}
a_w(u,v)=-\sum_{\al,\beta\in \Phi^+,~w=s_\al s_\beta} k_\al k_\beta (\al(u)\beta(v)-\al(v)\beta(u)).
\end{equation}
Therefore $a_w=0$ unless $w$ is the product of two distinct reflections. In particular, $1\notin W(\fa)$ and $\kappa_1=0$. Theorem \ref{t:D2} specializes to the formula for $\C D^2$ from \cite[Theorem 3.5]{BCT}:
\begin{equation}
\begin{aligned}
\Omega_\bH&=\sum_{i}v_i v^i+\sum_{\al,\beta\in\Phi^+} k_\al k_\beta \langle\al^\vee,\beta^\vee\rangle s_\al s_\beta \in Z(\bH);\\
\Omega_{\wti W,\ba}&=-\sum_{\al,\beta\in\Phi^+} |\al^\vee||\beta^\vee| \wti s_\al \wti s_\beta\in Z(\bC[\wti W]), \text{ where }\wti s_\al=\frac 1{|\al^\vee|}\al^\vee\in \Pin(V).\\
\end{aligned}
\end{equation}
 The representation theory of the pin cover $\wti W$ of the Weyl group that appears for the graded affine Hecke algebra was studied in relation with the Dirac operator and the geometry of the nilpotent cone in \cite{Ci,CH,CT,Cha}.

\smallskip

In later sections, we will discuss the case of rational Cherednik algebras, which form a different class of Drinfeld Hecke algebras.

\section{Vogan's Dirac morphism}

In this section, we construct an algebra homomorphism analogous to that in Vogan's conjecture in Dirac cohomology for $(\fg,K)$-modules, proved by Huang-Pand\v zi\'c \cite{HP}. In the setting of Lusztig's graded Hecke algebra, this was formulated and proved in \cite{BCT}.

\subsection{The linear map $d$} Recall the $\bZ/2\bZ$-grading of $C(V)$ by $C(V)=C(V)_0\oplus C(V)_1$ and the algebra automorphism  $\ep$  of $C(V)$ that equals $\Id$ on $C(V)_0$ and $-\Id$ on $C(V)_1$, extended to an automorphism of $\bfH\otimes C(V)$ by making it the identity on $\bfH.$ Define the linear map
\begin{equation}
d: \bfH\otimes C(V)\to \bfH\otimes C(V),
\end{equation}
by setting
\begin{equation}
d(a)=\C D a-\ep(a)\C D,\ a\in \bfH\otimes C(V).
\end{equation}

\begin{lemma}\label{l:der}
The map $d$ is an odd derivation, i.e.,
\begin{equation}
d(ab)=d(a)b+\ep(a) d(b),\quad a,b\in \bfH\otimes C(V).
\end{equation}
\end{lemma}

\begin{proof}
We verify the claim directly from the definition:
\begin{align*}
d(ab)&=\C D ab-\ep(ab)\C D=(\C D a -\ep(a) \C D) b + \ep(a) (\C D b -\ep(b)\C D)\\
&=d(a)b+\ep(a) d(b).
\end{align*}
\end{proof}

\begin{lemma}\label{l:W-ker}
$\Delta(\bC[\wti W])\subset \ker d.$
\end{lemma}

\begin{proof}
For every $\wti w\in \wti W$, by Lemma \ref{l:D-inv}, $\Delta(\wti w)\C D=\C D \ep(\wti w)$ in $\bfH\otimes C(V)$. The claim then follows.
\end{proof}

In the notation of Theorem \ref{t:D2}, define 
\begin{equation}
\wti \Omega_\bfH=\Omega_\bfH\otimes 1-1\otimes \frac 12\kappa_1\in (\bfH\otimes C(V))^{\wti W}.
\end{equation}
Define the subalgebra:
\begin{equation}\label{e:alg-A}
\bfA=Z_{\bfH\otimes C(V)}(\wti \Omega_{\bfH}).
\end{equation}
\begin{remark}If $1\notin W(\ba)$, then $\bfA=\bfH\otimes C(V).$
\end{remark}
In general, since $\wti\Omega_{\bfH}$ is $\wti W$-invariant, we have $\Delta(\bC[\wti W])\subset \bfA.$ We also have $\C D\in\bfA$ because of Theorem \ref{t:D2}. Notice that $\C D$ interchanges the trivial and $\det$ $\wti W$-isotopic spaces. Denote by $\bfA^{\wti W,\det}$ the $\det$-isotypic subspace. Restrict the map $d$ to
\begin{equation}
d_\triv: \bfA^{\wti W}\to \bfA^{\wti W,\det},\quad d_{\det}: \bfA^{\wti W,\det}\to \bfA^{\wti W}.
\end{equation}
Of course $d_\triv, \ d_{\det}$ are also an odd derivations. Then Lemma \ref{l:W-ker} implies:
\begin{equation}\label{e:W-ker}
\Delta(\bC[\wti W]^\wti W)\subset \ker d_\triv.
\end{equation}

The reason to restrict to $\bfA^{\wti W}$ is because of the following lemma.

\begin{lemma}\label{l:d2=0}
$d_\triv d_{\det}=d_{\det} d_\triv=0.$
\end{lemma}
\begin{proof}
We have $d^2(a)=\C D^2 a-a\C D^2$ for all $a\in \bfH\otimes C(V).$ If $a\in \bfA^{\wti W}$ or $a\in \bfA^{\wti W,\det}$, then $a$ commutes with both $\Omega_\bfH$ and $\Delta(\Omega_{\wti W,\ba})$, and therefore, by Theorem \ref{t:D2}, with $\C D^2.$
\end{proof}

The main result is a simple description of $\ker d_\triv$.

\begin{theorem}\label{t:ker-dW} The kernel of $d_\triv$ equals:
$$\ker d_\triv=\im d_{\det}\oplus \Delta(\bC[\wti W]^{\wti W}).$$
\end{theorem}

The proof will be obtained inductively from a reduction to the associated graded algebra. Recall the filtration on $\bfH$, obtained by giving degree $0$ to $\bC[W]$ and degree $1$ to the elements in $V$:
  $\bfH^0\subset \bfH^1\subset\dots\subset \bfH^n\subset\dots$. Let $\gr(\bfH)$ be the associated graded algebra. The PBW assumption on $\bfH$ means that
\begin{equation}\label{e:assoc-graded}
\gr(\bfH)\cong \bfH_{0}=S(V)\rtimes \bC[W].
\end{equation}
We consider the induced filtration on $\bfH\otimes C(V)$, in other words, the filtration is $(\bfH^n\otimes C(V))$. The associated graded object is then $\bfH_0\otimes C(V).$

It is important to observe that, by Proposition \ref{p:Omega-central}, commutation with $\wti \Omega_\bfH$ acts on the subspaces $\bfH^n\otimes C(V)$. This is because $[\Omega_\bfH,V]\subset V$ and $\Omega_\bfH$ is $W$-invariant,  therefore 
\begin{equation}
[\Omega_\bfH, \bfH^n\otimes C(V)]\subset \bfH^n\otimes C(V).
\end{equation}
Define the filtration on $\bfA$ by $\bfA^n=\bfA\cap \bfH^n\otimes C(V)$, and let $\gr(\bfA)\subset  \bfH_{0}$ be the associated graded object.

Notice that $d$ preserves the filtration $\bfH^n\otimes C(V),$ in fact, $d(\bfH^n\otimes C(V))\subseteq \bfH^{n+1}\otimes C(V).$ Denote the resulting maps by:
\begin{equation}
\bar d: \gr(\bfH\otimes C(V))\to \gr(\bfH\otimes C(V)),\quad \bar d_\triv: \gr(\bfA^{\wti W})\to \gr(\bfA^{\wti W,\det}),
\end{equation}
and similarly $\bar d_{\det}$.

\subsection{Semidirect products} One analyzes first the maps on the associated graded objects. Given (\ref{e:assoc-graded}), we are in the situation of $\bfH_{0}=S(V)\rtimes \bC[W]$ with the linear map $\bar d: \bfH_{0}\otimes C(V)\to \bfH_{0}\otimes C(V)$.  The following results, which generalize to semi direct products the case of a symmetric algebra from \cite{HP}, are proved in \cite{BCT} in the setting of the graded affine Hecke algebra. The statements and and the proofs, based on the cohomology of the usual Koszul complex, apply to this setting. We present and simplify the arguments here for the benefit of the reader.

\begin{lemma}
The map $\bar d:\bfH_0\otimes C(V)\to \bfH_0\otimes C(V)$ has the following properties:
\begin{enumerate}
\item[(i)] $\bar d$ is an odd derivation;
\item[(ii)] $\bar d^2=0$;
\item[(iii)] $\bar\Delta(\bC[\wti W])\subset \ker\bar d$;
\item[(iv)] $\im\bar d\cap \bar\Delta(\bC[\wti W])=0.$
\end{enumerate} 
\end{lemma}

\begin{proof} (i) follows from Lemma \ref{l:der}. For (ii), notice that $\bar d^2(a)=[\bar D^2,a]=[-\Omega_{\bfH_0}\otimes 1,a]=0$, since $\Omega_{\bfH_0}\in Z(\bfH_0)$ by Proposition \ref{p:Omega-central}.  Part (iii) follows from Lemma \ref{l:W-ker}. For (iv), notice that $$\bar d(\bfH_0^n\otimes C(V))\subset \bfH_0^{n+1}\otimes C(V),$$
and therefore $\im\bar d\subseteq \bigoplus_{n\ge 1} \bfH_0^n\otimes C(V).$ 
\end{proof}

Thus, we have $\im\bar d\oplus \bar\Delta(\bC[\wti W])\subseteq \ker \bar d.$
The proof of the following proposition in \cite{BCT} giving the opposite inclusion is based on the following fact.
The subalgebra $S(V)\otimes C(V)\subset \bfH_0\otimes C(V)$ is preserved by $\bar d$. Denote the restriction of $\bar d$ to this subalgebra by $\bar d'$. By \cite[Lemma 4.1]{HP}, $\bar d'$ is essentially the differential in the classical Koszul resolution and therefore 
\begin{equation}\label{e:koszul}
\ker \bar d'=\im \bar d'\oplus \bC.
\end{equation}

\begin{proposition}[cf. {\cite[Proposition 4.14]{BCT}}]\label{p:BCT} We have  $\ker \bar d=\im\bar d\oplus \bar\Delta(\bC[\wti W]).$
\end{proposition}

\begin{proof}
Let $wf\otimes g$ be a simple tensor in $\bfH_0\otimes C(V)$, where $f\in S(V)$ and $g\in C(V).$ Let $\wti w\in \wti W$ be such that $p(\wti w)=w$ and denote $g'=\wti w^{-1}\cdot g\in C(V)$. The derivation property of $\bar d$ implies
\begin{equation}
\bar d(wf\otimes g)=\bar d((p(\wti w)\otimes \wti w)(f\otimes \wti g')=d(p(\wti w)\otimes \wti w) (f\otimes g')+\det(p(\wti w)) d(f\otimes g')=d'(f\otimes g''),
\end{equation}
where $g''=\det(p(\wti w)) g'$. The claim follows then from the formula for $\ker \bar d'$.
\end{proof}
Restricting to $\wti W$-invariants, Proposition \ref{p:BCT} also gives:
\begin{equation}\label{dbar-inv}
\ker \bar d_\triv=\im\bar d_{\det}\oplus \bar\Delta(\bC[\wti W]^{\wti W}).
\end{equation}

\subsection{Induction}\label{s:induction} The proof of Theorem \ref{t:ker-dW} follows by a well-known induction using the filtration on $\bfA^{\wti W}.$ The argument in the setting of graded affine Hecke algebra is presented in \cite[\S3.3]{COT}  and can be applied verbatim to this setting. We need to record it however because it will be used to prove later a refinement in the case of rational Cherednik algebras at $t=0$.

By Lemma \ref{l:d2=0} and (\ref{e:W-ker}), we have $\Delta(\bC[\wti W]^{\wti W})+\im d_{\det}\subseteq \ker d_\triv.$ In fact, the sum in the left hand side is direct. To see this, notice that for every $a\in \im d_{\det}$, the term of top degree in $a$ with respect to the filtration  ($\bfH^n\otimes C(V))$ can be regarded as an element in $\im \bar d_{\det}$ with the same degree. In particular, if $a$ is in $\Delta(\bC[\wti W]^{\wti W})\cap\im d_{\det}$, then $a$ itself (having degree $0$) can be regarded as an element of $\Delta(\bC[\wti W]^{\wti W})\cap\im \bar d_{\det}$. Thus $a=0$ by (\ref{dbar-inv}).

We need to prove the inclusion $\ker d_\triv\subseteq \im d_{\det}\oplus \Delta(\bC[\wti W]^{\wti W}).$ Let $a\in \ker d_\triv$ be given and we may assume that $a\in \bfA^n.$ Taking graded objects, $\bar d_\triv(\bar a)=0$ in $\gr(\bfA).$ By (\ref{dbar-inv}), there exists $\bar b \in (\gr(\bfA)^{n-1})^{\wti W,\det}$ and $s\in\bC[\wti W]^{\wti W}$ such that 
\[\bar a=\bar d_{\det}\bar b+\bar\Delta(s).\]
Choose $b\in (\bfA^{n-1})^{\wti W,\det}$ such that $\bar b$ is the image of $b$ in the associated graded object: if $\bar b=\sum_w w v_{w,1}\dots v_{w,n-1}\otimes g_w$ with $v_{w,i}\in V$ and $g_w\in C(V)$, then set $b=\sum_w w v_{w,1}\dots v_{w,n-1}\otimes g_w\in \bfA^{n-1}.$ Then 
\[\overline{a-d_{\det} b-\Delta(s)}=\bar a-\bar d_{\det}\bar b-\bar\Delta(s)=0,\]
which means that $a-d_{\det} b-\Delta(s)\in (\bfA^{n-1})^{\wti W}.$ Since
\[d_\triv(a-d_{\det} b-\Delta(s))=d_\triv(a)-d^2(b)-d_{\triv}(\Delta(s))=0-0-0=0,\]
the claim follows by induction.

\subsection{Vogan's Dirac homomorphism} 
Since $d_\triv$ is an odd derivation, it follows that if $a,b\in \ker d_\triv$, then also $ab\in\ker d_\triv.$ In other words, $\ker d_\triv$ is an algebra. The following result is a slight modification of the one for real semisimple Lie algebras \cite{HP} or graded affine Hecke algebras \cite{BCT}.

\begin{theorem}\label{t:vogan}
The projection $\zeta: \ker d_\triv\to \bC[\wti W]^{\wti W}$ defined by Theorem \ref{t:ker-dW} is an algebra homomorphism.
\end{theorem}

\begin{proof}
Let $z_1,z_2\in \ker d_\triv$ be given. By Theorem \ref{t:ker-dW}, there exist $a_1,a_2\in \bfA^{\wti W}$ such that $z_i=\Delta(\zeta(z_i))+d_{\det}(a_i).$ Then
\begin{align*}
z_1z_2&=\Delta(\zeta(z_1)\zeta(z_2))+d_{\det}(a_1)\Delta(\zeta(z_2))+\Delta(\zeta(z_1))d_{\det}(a_2)+d_{\det}(a_1)d_{\det}(a_2)\\
&=\Delta(\zeta(z_1)\zeta(z_2))+d_{\det}(a_1\Delta(\zeta(z_2))+\Delta(\zeta(z_1))a_2) +d_{\det}(a_1d_{\det}(a_2)),
\end{align*}
where we have used the commutation relation between $\C D$ and $\Delta(\bC[\wti W])$ and that $d_{\det}$ is a derivation. Now apply Theorem \ref{t:ker-dW} to $z_1z_2$, and the claim follows.
\end{proof}

Since the image of $\zeta$ is an abelian algebra, the homomorphism $\zeta$ must factor through to the abelianization $(\ker d_\triv)^\ab=\ker d_\triv/[\ker d_\triv,\ker d_\triv].$

\begin{corollary}
The map $\zeta$ gives an algebra homomorphism $\zeta: (\ker d_\triv)^\ab\to \bC[\wti W]^{\wti W}$.
\end{corollary}

\begin{example}
We remark that 
\begin{equation}
\wti \Omega_\bfH\in \ker d_\triv.
\end{equation}
Indeed, since $\wti\Omega_\bfH=-\C D^2+\Delta(\Omega_{\wti W,\ba})$ and $\C D^2$ is $\Delta(\wti W)$-invariant, we have that $\wti\Omega_\bfH$ commutes with $\C D.$ Moreover, $\wti\Omega_\bfH\in (\bfH\otimes C(V)_0)^{\wti W}$, and therefore $d_\triv(\wti\Omega_\bfH)=0.$
The formula for $\C D^2$ implies that $\zeta(\wti\Omega_\bfH)=\Omega_{\wti W,\ba}$.
\end{example}

One uses Theorem \ref{t:vogan} as follows. Suppose $\C B$ is a (finitely generated) abelian subalgebra of $\bfA\cap \bfH\otimes C(V)_0$ such that $\wti\Omega_{\bfH}\in\C B\subset \ker d_\triv$. This means that every element in $\C B$ commutes with $\C D$.

\begin{remark}
When $a_1=0$, one may choose $\C B=Z(\bfH)\otimes 1$.
\end{remark}

Then the homomorphism in Theorem \ref{t:vogan} defines a morphism
\begin{equation}\label{e:spec-morphism}
\zeta^*: \Irr(\wti W)=\Spec \bC[\wti W]^{\wti W}\to \Spec \C B.
\end{equation}

Later in the paper, we will discuss the morphism $\zeta^*$ in the case of rational Cherednik algebras.

\subsection{Dirac cohomology}
We can define the notion of Dirac cohomology in this setting. The Clifford algebra $C(V)$ is central simple when $\dim V$ is even and it has a unique complex simple module $S$. When $\dim V$ is odd, the even part $C(V)_0$ is central simple with unique simple module $S$ which can be extended in two non isomorphic ways, $S^+$ and $S^-$ to $C(V).$ The restrictions of the spin modules to $\wti W$ give $\wti W$-representations which are related as follows
\begin{equation}
S\otimes \det\cong S,\ S^+\otimes \det\cong S^-.
\end{equation}
\begin{definition}\label{d:coh}
Let $X$ be a finitely generated $\bfH$-module. We say that $X$ is $\Omega_\bfH$-admissible if $X$ has a decomposition into $\Omega_\bfH$-generalized eigenspaces:
\begin{equation}
X=\bigoplus_{\lambda\in \bC} X_\lambda,
\end{equation}
such that each $X_\lambda$ is finite dimensional. Since $\Omega_\bfH$ is $W$-invariant, each $X_\lambda$ is a finite dimensional $W$-representation.

Let $X$ be an $\Omega_\bH$-admissible $\bfH$-module and let $\C S\in \{S,S^+,S^-\}$ be a spin module. The Dirac operator of $X$ (and $\C S$) is 
\begin{equation}
D_X: X\otimes \C S\to X\otimes \C S, 
\end{equation}
given by the action of the Dirac element $\C D$. The Dirac cohomology of $X$ (and $\C S$) is
\begin{equation}
H_D(X)=\ker D_X/\ker D_X\cap \im D_X.
\end{equation}
\end{definition}

\begin{lemma}
Suppose $X$ is an $\Omega_\bfH$-admissible module. Then $H_D(X)$ is a finite dimensional $\wti W$-representation (or zero).
\end{lemma}

\begin{proof}
The fact that $H_D(X)$ is a $\wti W$-representation follows from Lemma \ref{l:D-inv}. Suppose $\wti\sigma$ is a $\wti W$-representation that occurs in $H_D(X)$. Let $0\neq \wti x\in \ker D_X$ be an element of the $\wti\sigma$-isotypic component lying in $X_\lambda \otimes\C S$ for an $\Omega_\bfH$-eigenspace $X_\lambda$, $\lambda\in\bC$.  

Suppose $\C S=\oplus_{j=1}^\ell\wti\mu_j$ is the decomposition of $\C S$ into irreducible $\wti W$-representations. The element $\kappa_1\in C(V)$ is $\wti W$-invariant, thus $\kappa_1$ acts on each $\mu_j$ acts by a scalar.  Since $1\otimes \frac 12\kappa_1$ commutes with both $\Omega_{\bfH}\otimes 1$ and with $\Delta(\bC[\wti W])$, we may assume without loss of generality that $(1\otimes \frac 12\kappa_1)$ acts on $\wti x$ by a scalar $m_1.$ 

Since $D_X^2(\wti x)=0$, Theorem \ref{t:D2} gives:
\[(\Omega_{\bfH}\otimes 1) \wti x= (m_1 +N_{\ba}(\wti\sigma) )\wti x,\]
where $N_{\ba}(\wti\sigma)$ is the scalar by which  the central element $\Omega_{\wti W,\ba}$ acts on $\wti\sigma.$
Thus $\lambda=m_1+N_{\ba}(\wti\sigma).$ The scalar $m_1$ is depends only on the decomposition of $\C S$ into irreducible $\wti W$-representations, and thus there are only finitely many possibilities. Since there also only finitely many different $\wti\sigma$'s, it follows that there are only finitely many possible $\lambda$'s that contribute to $H_D(X)$. Hence $H_D(X)$ is finite dimensional.
\end{proof}

The main idea behind Dirac cohomology applied to this setting is the following result. Recall the morphism
$\zeta^*:\Irr(\wti W)\to \Spec \C B$ for an abelian subalgebra 
\begin{equation}\label{e:B}
\C B\subset (\bfH\otimes C(V))^{\wti W}
\end{equation}
 that satisfies
\begin{enumerate}
\item $\wti\Omega_\bfH\in \C B$;
\item $[\C D,\C B]=0.$
\end{enumerate}
Notice that $H_D(X)$ is a $\bC[\wti W]\otimes \C B$-module. 

\begin{theorem}\label{t:vogan-conj}
Let $X$ be an $\Omega_{\bfH}$-admissible $\bfH$-module. Let $\C B$ be an algebra as in (\ref{e:B}). Suppose $H_D(X)\neq 0$. If there is a nonzero $\wti\sigma\otimes\chi$-isotypic component of $\bC[\wti W]\otimes \C B$ in $H_D(X)$ for $\wti\sigma\in\Irr(\wti W)$ and $\chi\in\Spec \C B$, then
\[\chi=\zeta^*(\wti\sigma).
\]
\end{theorem}

\begin{proof}
Suppose $\wti x\neq 0$ is an element in the $\wti\sigma$-isotypic component in $H_D(X)$ such that $\C B$ acts on $x$ by $\chi\in\Spec\C B$. 
Let $0\neq z\in \C B$ be given. By Theorem \ref{t:vogan}, there exists $a\in \bfH\otimes C(V)$ such that
\[z=\Delta(\zeta(z))+\C D a-\ep(a)\C D.\]
Apply this to $\wti x$. We find:
\[\chi(z)\wti x=\wti\sigma(\zeta(z)) x+ \C D a \wti x-\ep(a) \C D \wti x=\sigma(\zeta(x)) x+\C D a,\]
where $\wti\sigma(\zeta(z))$ and $\chi(z)$ are scalars.Then  
\[(\chi(z)-\sigma(\zeta(z))\wti x=\C D a\wti x.\]
Since the right hand side is in $\im D_X$ and the left hand side is in $H_D(X)$, it follows that
\[(\chi(z)-\sigma(\zeta(z))\wti x=0,\]
and since $\wti x\neq 0$, we get $\chi(z)=\sigma(\zeta(z)$ for all $z\in\C B$, or in other words, $\chi=\zeta^*(\wti\sigma).$
\end{proof}

\begin{remark}
If $a_1=0$, then we choose $\C B=Z(\bfH)\otimes 1$ and then Theorem \ref{t:vogan-conj} says that the central character of $X$ is uniquely determined by $H_D(X)$.
\end{remark}

\section{Symplectic reflection algebras}
We apply the theory to the class of symplectic reflection algebras introduced by Etingof and Ginzburg \cite{EG}.

\subsection{Definition} Let $V$ be a $2n$-dimensional complex vector space carrying a non degenerate symplectic $2$-form $\omega$ and let $Sp(V)$ be the corresponding symplectic group. An element $s\in Sp(V)$ is called a symplectic reflection if $\mathsf{rk}(\Id_V-s)=2$. In that case $V=\ker(\Id_V-s)\oplus \im(\Id_V-s)$ is an $\omega$-orthogonal decomposition. Denote by $\omega_s$ the skew-symmetric form that has $\ker(\Id_V-s)$ as its radical, and it equals $\omega$ on $\im(\Id_V-s)$.

Let $W\subset Sp(V)$ be a finite symplectic reflection group, i.e., a finite subgroup generated by  symplectic reflections. Let $\C R\subset W$ denote the set of symplectic reflections in $W.$ (We have reserved the notation $S$ for the spin module of the Clifford algebra.) Let $c:\C R\to\bC$ be a $W$-invariant parameter function and $t\in\bC$ a constant.

\begin{definition}[{\cite[Theorem 1.3]{EG}}] The symplectic reflection algebra $\bfH_{t,c}$ associated to the data above is the quotient of the semi direct product $T(V)\rtimes \bC[W]$ by the relations:
\begin{equation}
[u,v]=t\omega(u,v)+\sum_{s\in S} c_s\omega_s(u,v) s.
\end{equation}
\end{definition}
These algebras have the PBW property. In fact, they are particular cases of Drinfeld's Hecke algebras $\bfH_\ba$ (Definition \ref{d:Drinfeld}) if we take the family of forms $\ba$ to equal:
\begin{equation}
a_w=\begin{cases}\omega,&w=1,\\
\omega_s,&w=s\in \C R,\\
0,&\text{otherwise.}\end{cases}
\end{equation}

\subsection{Examples} The classification of indecomposable symplectic reflection groups is known, see \cite{Co}. In particular, there are two important families of symplectic reflection groups:

\begin{enumerate}
\item (Complex reflection groups). Let $\fh$ be a finite dimensional $\bC$-vector space and $W\subset GL(\fh)$ a complex reflection group. Let $V=\fh+\fh^*$ with the standard symplectic form
\begin{equation}
\omega(y_1+x_1,y_2+x_2)=x_2(y_1)-x_1(y_2),\quad x_1,x_2\in\fh^*,~y_1,y_2\in\fh.
\end{equation}
The group $W$ acts diagonally on $V=\fh+\fh^*.$
\item (Wreath products). Let $\omega$ be a non degenerate symplectic form on $\bC^2$ and $\Gamma\subset Sp(2,\bC)$ be a finite group. Take $V=\bigoplus_{i=1}^n \bC^2$ with symplectic form $\omega$ induced from the symplectic form on $\bC^2$. Let $W$ be the wreath product $W=S_n\wr \Gamma$, where $S_n$ is the symmetric group acting on $V$: the $i$-th copy of $\Gamma$ acts on the $i$-th copy of $\bC^2$ and $S_n$ acts by permuting the copies of $\bC^2$.
\end{enumerate}

Assume in addition that there exists a non degenerate $W$-invariant symmetric bilinear form $\<~,\>$ on $V$.  Such is always the case when the symplectic reflection algebra comes from complex reflection groups. Define the symmetric form coming from the pairing of $\fh$ and $\fh^*$, i.e.:
\begin{equation}
\<x,x\>=0,\ \<y,y\>=0,\ \<x,y\>=\<y,x\>=x(y),
\end{equation}
for all $x\in\fh^*$ and $y\in\fh$. This form is $W$-invariant. This is the case of rational Cherednik algebras that we discuss next.

\subsection{Rational Cherednik algebra}
As before, let $\fh$ be a dimensional $\bC$-vector space, denote by $\fh^*$ its dual, and $V=\fh+\fh^*.$ Let $\langle~,~\rangle:V\times V\to\bC$ be the natural bilinear symmetric pairing defined in the previous subsection.  Let $W\subset GL(\fh)$ be a complex reflection group with set of pseudo-reflections $\C R$ acting diagonally on $V$. 

For every reflection $s\in\C R$, the spaces $\im(\Id_V-s)|_{\fh^*}$ and $\im(\Id_V-s)|_{\fh}$ are one-dimensional. Choose $\al_s$ and $\al_s^\vee$ nonzero elements in $\im(\Id_V-s)|_{\fh^*}$ and $\im(\Id_V-s)|_{\fh}$, respectively. Then there exists $\lambda_s\in \bC$, $\lambda_s\neq 1$ a root of unity, such that
\begin{equation}
s(\al_s^\vee)=\lambda_s \al_s^\vee,\ s(\al_s)=\lambda_s^{-1} \al_s.
\end{equation}
(In the case when $W$ is a finite reflection group, $\lambda_s=-1$.)
For every $v\in V$ such that $\<v,v\>\neq 0$, denote by $s_v$ the reflection in the hyperplane perpendicular to $v$. The reflection $s_v$ is given by:
\[s_v(u)=u-\frac 2{\langle v,v\rangle}\langle u, v\rangle v,\ u\in V.\]

\begin{lemma}\label{l:s-split}
Let $\sqrt{\lambda_s}$ be a square root of $\lambda_s$. Then $s=s_{v_s}s_{v_s'} \in O(V)$, where  $v_s=\sqrt{\lambda_s}\al_s^\vee+\al_s$ and $v_s'=\al_s^\vee+\sqrt{\lambda_s}\al_s$. 
\end{lemma}

\begin{proof}
Straightforward.
\end{proof}

\begin{definition}
The rational Cherednik algebra $\bfH_{t,c}$ associated to $\fh,W$, the parameter $t\in\bC$, and the $W$-invariant parameter function $c:\C R\to\bC$ is the quotient of $T(V)\rtimes W$ by the relations:
\begin{enumerate}
\item $[y_1,y_2]=0$, $[x_1,x_2]=0$, for all $y_1,y_2\in \fh$, $x_1,x_2\in\fh^*$;
\item $\displaystyle{[y,x]=t\langle y,x\rangle-\sum_{s\in\C R} c_s\frac{\langle y,\al_s\rangle\langle\al_s^\vee,x\rangle}{\langle\al_s^\vee,\al_s\rangle} s},$ for all $y\in\fh,$ $x\in \fh^*.$
\end{enumerate}
\end{definition}
To translate to this setting the results proved in the general setting of Drinfeld's algebra $\bfH_\ba$, notice that the only nonzero values of the skew-symmetric forms $a_w$ are given by:
\begin{equation}
a_1(y,x)=t\langle y,x\rangle,\ a_s(y,x)=-c_s\frac{\langle y,\al_s\rangle\langle\al_s^\vee,x\rangle}{\langle\al_s^\vee,\al_s\rangle},
\end{equation}
for $y\in\fh$, $x\in\fh^*$.

Let $\{y_i\}$ be a basis of $\fh$ and $\{x_i\}$ the dual basis of $\fh^*$. The dual bases of $V$ are then $(v_i)=((y_i),(x_i))$ and $(v^i)=((x_i),(y_i)).$

The element $\bh$ defined in (\ref{e:bh}) equals:
\begin{equation}
\bh=\sum_i (y_ix_i+x_iy_i).
\end{equation}
We compute the scalars $e_w$ that were defined in (\ref{e:w-tilde}). Writing $s=s_vs_{v'}$ as in Lemma \ref{l:s-split}, we have
\[
e_s=\frac{a_s(v_s,v_s') \langle v_s,v_s'\rangle}{\langle v_s,v_s\rangle\langle v_s',v_s'\rangle-\langle v_s,v_s'\rangle^2}.
\]
Computing directly, we find
\begin{equation}
\begin{aligned}
a_s(v_s,v_s')=c_s(\lambda_s-1)\langle \al_s^\vee,\al_s\rangle,\ \langle v_s,v_s'\rangle=(\lambda_s+1)\langle\al_s^\vee,\al_s\rangle,\\
 \langle v_s,v_s\rangle=\langle v_s',v_s'\rangle=2\sqrt\lambda\langle\al_s^\vee,\al_s\rangle,
\end{aligned}
\end{equation}
and therefore 
\begin{equation}
e_s=c_s\frac{\lambda_s+1}{1-\lambda_s}.
\end{equation}
Thus the element 
$\Omega_{\bfH_{t,c}}$ defined in (\ref{e:Omega}) equals
\begin{equation}
\Omega_{\bfH}=\sum_i (y_ix_i+x_iy_i)-\sum_{s\in \C R} c_s\frac{\lambda_s+1}{1-\lambda_s} s\in \bfH_{t,c}^W.
\end{equation}
Using the commutation relation between $x_i,y_i$, one shows easily that
\[\sum_{i} y_ix_i=\sum_{i}x_iy_i+nt-\sum_s c_s s,\]
and thus
\begin{equation}\label{e:Omega-cherednik}
\Omega_\bfH=2\sum_i x_iy_i+nt-2\sum_{s\in \C R}\frac {c_s}{1-\lambda_s} s.
\end{equation}
By Proposition \ref{p:Omega-central} applied to this setting,
\begin{equation}\label{e:h-grading}
[\Omega_{\bfH},x]=2tx,\ [\Omega_{\bfH},y]=-2ty,\quad \text{for all }x\in\fh^*, y\in\fh.
\end{equation}
\begin{remark}
The element $\Omega_\bfH$ is of course known in the literature of rational Cherednik algebra. It essentially equals twice the well-known grading  as used in \cite{GGOR} for example. It is interesting that the same element appears naturally in the present Dirac setting.
\end{remark}

\subsection{The Clifford algebra} 
Recall $C(V)$, the complex Clifford algebra defined by $V$ and $\langle~,~\rangle$.  In terms of the $x_i,y_i$'s the relations are:
\begin{equation}
x_i\cdot x_j=-x_j\cdot x_i,\ y_i\cdot y_j=-y_j\cdot y_i,\ x_i\cdot y_j+y_j\cdot x_i=-2\delta_{i,j}.
\end{equation} 

The spin $C(V)$-module $S$ is realized on the vector space $\bigwedge \fh$ with the action:
\begin{align}
y\cdot (y_1\wedge \dots\wedge y_k)&=y\wedge y_1\wedge\dots\wedge y_k,\quad y\in \fh;\\
x\cdot (y_1\wedge\dots\wedge y_k)&=2\sum_{i}(-1)^i\langle y_i,x\rangle y_1\wedge\dots\wedge\hat y_i\wedge\dots\wedge y_k.
\end{align}

The element $\kappa_1$ from (\ref{e:kappa1}) becomes
\begin{equation}\label{e:kappa-cherednik}
\frac 12\kappa_1=\frac t2\sum_{i} (x_i y_i-y_ix_i)=t(\bom+n)\in C(V),\text{ where }\bom=\sum_ix_iy_i.
\end{equation}
From Lemma \ref{l:kappa-grading}, for every $x\in\fh^*$ and $y\in\fh$, we have the following relations in $C(V)$:
\begin{equation}\label{e:kappa-grading} 
[\frac12 \kappa_1,x]=-2tx\text{ and } [\frac 12\kappa_1,y]=2ty.
\end{equation}

\begin{lemma} The element $\frac 12\kappa_1$ acts by $t(-n+2\ell)\cdot\Id$ on $\bigwedge^\ell\fh$ in the spin module $S$.
\end{lemma}

\begin{proof}
Straightforward, using the definition of the action of $C(V)$ on $S$.
\end{proof}

\subsection{Pin cover of $W$} Following (\ref{e:w-tilde}), for every $s\in S$, define
\[
\wti s=\frac 1{|v_s||v_s'|}v_s\cdot v_s'\in \wti W,\quad c_{\wti s}=\frac{a_s(v_s,v_s')}{\langle v_s,v_s\rangle\langle v_s',v_s'\rangle-\langle v_s,v_s'\rangle^2}|v_s||v_s'|,
\]
 which simplifies to
 \[c_{\wti s}\wti s=-c_s\left(\frac{1}{\langle\al_s^\vee,\al_s\rangle}\al_s^\vee\al_s+\frac 2{1-\lambda_s}\right).
 \]
 Define for every $s\in \C R$:
 \begin{equation}\label{e:tau-s}
\tau_s=\frac{1-\lambda_s}{2\langle\al_s^\vee,\al_s\rangle}\al_s^\vee\al_s+1\in C(V).
\end{equation}
so that
\begin{equation}\label{e:s-tilde}
c_{\wti s}\wti s=\frac {2c_s}{\lambda_s-1}\tau_s.
\end{equation}

\begin{lemma}\label{l:tau-al}
The elements $\tau_s$ have the properties:
\begin{enumerate}
\item $\tau_s\cdot u\cdot \tau_s^{-1}=s(u),\text{ for all }u\in V.$
\item $\ep(\tau_s)=\tau_s$ and $\tau_s^t=\lambda_s\tau_{s^{-1}}.$
\item $\tau_{s^{-1}}=\tau_s^{-1}.$
\end{enumerate}
\end{lemma}

\begin{proof}
(1) Since $\tau_s$ is a scalar times $v_s\cdot v_s'$, it follows that conjugation by $\tau_s$ in $C(V)$ equals
$\tau_s\cdot u\cdot \tau_s^{-1}=s_{v_s}s_{v'_s}(u)=s(u)$, for $u\in V.$

(2) The element $\tau_s$ is in $C(V)_0.$ An easy calculation gives $\tau_s^t=\lambda_s\left(\frac{1-\lambda_s^{-1}}{2\langle\al_s^\vee,\al_s\rangle}\al_s^\vee\al_s+1\right)$, and the claim follows since $\lambda_{s^{-1}}=\lambda_s^{-1}$ and $\al_{s^{-1}}=\al_s$, $\al_{s^{-1}}^\vee=\al_s^\vee.$

(3) This is an easy, direct calculation.
\end{proof}

\begin{remark}\label{r:pin-cover}
Recall the double cover $p:\mathsf{Pin}(V)\to O(V)$, $p(a)(v)=\ep(a)\cdot v\cdot a^{-1}$ and $\wti W=p^{-1}(W)\subset \mathsf{Pin}(V)$. Then Lemma \ref{l:tau-al} implies that
\begin{equation}
p^{-1}(s)=\{\lambda_s^{-1/2}\tau_s,-\lambda_s^{-1/2}\tau_s\}.
\end{equation}
The map $s\mapsto \tau_s$ extends to a homomorphism
\begin{equation}\label{l:W-embed}
\tau:W\to C(V)^\times.
\end{equation}
Indeed, notice first that by Lemma \ref{l:tau-al}(3), $\tau_{s^{-1}}=\tau_s^{-1}$. Secondly, if $s_1,s_2$ are two reflections in $\C R$, then $\tau_{s_1}\cdot \tau_{s_2}\cdot \tau_{s_1^{-1}}=\tau_{s_1}\cdot \tau_{s_2}\cdot \tau_{s_1}^{-1}=s_1(\tau_{s_2})$ by Lemma \ref{l:tau-al}(1). Finally, $s_1(\tau_{s_2})=\frac{1-\lambda_{s_2}}{2\langle\al_{s_2}^\vee,\al_{s_2}\rangle}s_1(\al_{s_2}^\vee)s_1(\al_{s_2})+1=\tau_{s_3}$, where $s_3=s_1s_2s_1^{-1}.$

From now on, define $\tau_w$ to be the image in $C(V)^\times$ of $w\in W$ under the map in (\ref{l:W-embed}).  For convenience, we may work with this embedding rather than $\wti W$ itself; however, we should note that, unlike the cover $\wti W$, the choice of map $\tau$ is not canonical.\end{remark}

Since $C(V)$ acts on $S$, we get an action of $\tau(W)$ on $S$. The following lemma describes the action. 

\begin{lemma}\label{l:W-action} For every $s\in \C R$, \begin{equation}\label{e:det-dual}
\tau_s\cdot(y_1\wedge\dots\wedge y_\ell)=s(y_1)\wedge\dots\wedge s(y_\ell).
\end{equation} 
Thus, the action of $\tau(W)$ on $S$ preserves each piece $\bigwedge^\ell\fh$ of $S$, where it acts as the natural action of $W$ on $\bigwedge^{\ell}\fh.$
\end{lemma}

\begin{proof}
This is a direct computation, using the action of $C(V)$ on $S$. Since the action is linear, it is sufficient to check it on any particular basis of $\fh$. So choose $y_1=\al_s^\vee$, and $y_2,\dots,y_n\in \ker(\Id_{\fh}-s)$.  Let $y_{i_1}\wedge\dots\wedge y_{i_\ell}$ be a basis vector in $\bigwedge^\ell \fh$, with $i_1<i_2<\dots<i_\ell$. There are two cases.

The first case is when $i_1\ge 2.$ Then:
\begin{align*}
(\al_s^\vee\al_s)&\cdot (y_{i_1}\wedge y_{i_2}\wedge\dots\wedge y_{i_\ell})=0,
\end{align*}
and so $\tau_s\cdot (y_{i_1}\wedge y_{i_2}\wedge\dots\wedge y_{i_\ell})= y_1\wedge y_{i_2}\wedge\dots\wedge y_{i_\ell}$.
On the other hand, 
\begin{align*}
s(y_{i_1})\wedge s(y_{i_2})\wedge\dots \wedge s(y_{i_\ell})&= y_{i_1}\wedge y_{i_2}\dots\wedge y_{i_\ell}.\end{align*}

In the second case, $i_1=1.$ Then:
\begin{align*}
(\al_s^\vee\al_s)\cdot (y_{1}\wedge y_{i_2}\wedge\dots\wedge y_{i_\ell})&=\al_s^\vee\cdot(-2\langle\al_s^\vee,\al_s\rangle  y_{i_2}\wedge\dots\wedge y_{i_\ell})\\
&=-2\langle\al_s^\vee,\al_s\rangle  (y_{1}\wedge y_{i_2}\wedge\dots\wedge y_{i_\ell}),
\end{align*}
and so $\tau_s\cdot (y_{1}\wedge y_{i_2}\wedge\dots\wedge y_{i_\ell})=\lambda_s (y_{1}\wedge y_{i_2}\wedge\dots\wedge y_{i_\ell}).$
On the other hand, now $s(y_{1})\wedge s(y_{i_2})\wedge\dots \wedge s(y_{i_\ell})=\lambda_s(y_{1}\wedge y_{i_2}\wedge\dots \wedge y_{i_\ell}).$

The claim follows.

\end{proof}

\subsection{The Dirac elements}
The Dirac element in $\bfH_{t,c}\otimes C(V)$ is:
\begin{equation}
\C D=\C D_x+\C D_y, \text{ where }\C D_x=\sum_{i} x_i\otimes y_i,\quad \C D_y=\sum_{i} y_i\otimes x_i.
\end{equation}

Let $\Delta: W\to \bfH_{t,c}\otimes C(V)$ denote the group homomorphism $w\mapsto w\otimes \tau_w$, see Remark \ref{r:pin-cover}. We also denote by $\Delta$ the map $\bC[W]\to \bfH_{t,c}\otimes C(V)$ that extends linearly $w\mapsto w\otimes \tau_w.$

\begin{proposition}\label{p:square-dirac}
The Dirac elements have the following properties in $\bfH_{t,c}\otimes C(V)$:
\begin{enumerate}
\item $\C D_x, \C D_y,$ and $\C D$ are $W$-invariant, i.e., invariant with respect to the conjugation action of $\Delta(W).$
\item $\C D_x^2=\C D_y^2=0.$
\item The square of the Dirac element $\C D$ equals:
\begin{equation}
\C D^2=-\wti\Omega_{\bfH}-\Delta(\Omega_{W,c}),
\end{equation}
 with
\begin{equation}
\wti\Omega_{\bfH}=\Omega_\bfH\otimes 1-1\otimes \frac 12\kappa_1\in (\bfH_{t,c}\otimes C(V))^W
\end{equation}
and 
 \begin{equation}\label{e:Omega-W-2}
 \Omega_{W,c}=\sum_{s\in \C R} \frac {2c_s}{1-\lambda_s} s\in \bC[W]^W.
 \end{equation}
 The formulas for $\Omega_\bfH\in \bfH^W$ and $\frac 12\kappa_1\in C(V)^W$  are given in (\ref{e:Omega-cherednik}) and (\ref{e:kappa-cherednik}), respectively.
\end{enumerate}
\end{proposition}

\begin{proof} Claim (1) follows immediately from Lemma \ref{l:tau-al} or from the general Lemma \ref{l:D-inv} since $\det_V(s)=1$. For (2):
\[\C D_x^2=\sum_{i,j} x_i x_j\otimes y_i y_j =\sum_{i<j} [x_i,x_j]\otimes y_i y_j=0,\]
and similarly for $\C D_y^2.$

Claim (3) is a particular case of Theorem \ref{t:D2} where we use (\ref{e:s-tilde}). 
\end{proof} 

\begin{remark}\label{r:commute-Omega} When $t=0$, we have $\wti\Omega_{\bfH_{0,c}}=\Omega_{\bfH_{0,c}}\otimes 1$ and $\Omega_{\bfH_{0,c}}$ is a central element in $\bfH_{0,c}.$

 When $t\neq 0$, the center of the algebra $\bfH_{t,c}$ consists of the scalars $\bC$ only, \cite[Proposition 7.2]{BG}. In this case, notice that $\wti\Omega_{\bfH_{t,c}}$ commutes with all elements of the form $Z_{\bfH_{t,c}}(\Omega_{\bfH_{t,c}})\otimes 1+1\otimes Z_{C(V)}(\frac 12\kappa_1)$. More interestingly, (\ref{e:h-grading}) and Lemma \ref{e:kappa-grading} imply that
\begin{equation}
[\wti\Omega_{\bfH_{t,c}},x\otimes y]=0\text{ and } [\wti\Omega_{\bfH_{t,c}},y\otimes x]=0,
\end{equation}
for all $y\in\fh$ and $x\in\fh^*.$
\end{remark}

\begin{remark}
Given Proposition \ref{p:square-dirac}, it is important to know the value of the scalar $N_c(\sigma)$ by which $\Omega_{W,c}$ acts in an irreducible representation $\sigma$ of $W$. When $W$ is a finite real reflection group and $c$ is constant, $\Omega_{W,c}=c\sum_{s\in\C R}s$ acts by
\begin{equation}
N_c(\sigma)=c(a(\sigma\otimes\det)-a(\sigma)),
\end{equation} 
where $a(\sigma)$ is Lusztig's $a$-invariant for $\sigma$. This fact was noticed empirically by Beynon-Lusztig \cite{BL} and proved uniformly by Opdam \cite{Op}.
\end{remark}

\subsection{Wreath products} We conclude the section with some remarks about symplectic reflection algebras for a wreath product group. We follow the definitions and notation of \cite{GG}. 

Let $\Gamma$ be a finite group acting on the $2$-dimensional space $L=\bC^2$. Choose a basis $\{x,y\}$ of $L$ and the symplectic form $\omega_L(x,y)=1$. Let $V=L^{\oplus n}$ with the symplectic form $\omega=\omega^{\oplus n}$. For every $u\in L$, denote by $u_i$ the embedding of $u$ into the $i$-th component of $V$. Denote $\Gamma_n=S_n\wr\Gamma.$ 
There are two types of symplectic reflections in $\Gamma_n$:
\begin{enumerate}
\item[$(\Gamma)$] $\gamma_i$, for every $i\in[1,n]$ and every $\gamma\in\Gamma\setminus\{1\}$;
\item[$(S)$] $s_{ij}\gamma_i\gamma_j^{-1}$, for all $i,j\in [1,n]$ and $\gamma\in\Gamma.$
\end{enumerate}

\begin{definition}[{\cite[Lemma 3.1.1]{GG}}] The symplectic reflection algebra $\bfH_{t,k,c}(\Gamma)$ associated to the wreath product $\Gamma_n$ and parameters $t$, $k$, and $\{c_\gamma: \gamma\in\Gamma\setminus\{1\}\}$ is the the quotient of $T(V)\rtimes \Gamma_n$ by the relations
\begin{enumerate}
\item[(R1)] For every $i\in[1,n]$,
\[ [x_i,y_i]=t\cdot 1+\sum_{\gamma\in\Gamma\setminus\{1\}}c_\gamma\gamma_i+\frac k2\sum_{j=1,j\neq i}^n\sum_{\gamma\in\Gamma} s_{ij}\gamma_i\gamma_j^{-1};\]
\item For every $u,v\in L$ and $i\neq j$
\[ [u_i,v_j]=-\frac k2 \sum_{\gamma\in\Gamma} \omega_L(\gamma(u),v) s_{ij}\gamma_i\gamma_j^{-1}.\]
\end{enumerate}

\end{definition}

To construct the Dirac operator, we need to endow $V$ with a $\Gamma$-invariant symmetric bilinear form. The natural construction would be to define the symmetric form $\langle x,y\rangle_L=1$ on $L$  and the symmetric form $\<~,~\>_L^{\oplus n}$ on $V$. If this is the case, then $\Gamma$ must be a subgroup of $SO(2)$, and so $\Gamma=\bZ/r\bZ$. In this situation, the symplectic reflection algebra for $\Gamma_n$ is isomorphic to the rational Cherednik algebra for $\Gamma_n$, so this is a particular case of the previous discussion. If we choose $\{x,y\}$ to be a basis of $L$ such that  every $\gamma\in \Gamma$, $\gamma\neq 1$ acts by
\begin{equation}
\gamma(x)=\lambda_\gamma x,\ \gamma(y)=\lambda_\gamma^{-1} y.
\end{equation}
then the
The Dirac element $\C D\in \bfH_{t,k,c}\otimes C(V)$ is 
\[ \C D=\sum_{i} (x_i\otimes y_i+y_i\otimes x_i).
\]
and its square equals:
\begin{equation}
\C D^2=-\Omega_{\bfH}+1\otimes \frac 12\kappa_1-\Delta(\Omega_{\Gamma_n,k,c}),
\end{equation}
 where
\begin{equation}
\begin{aligned}
\Omega_{\bfH}&=\sum_i(x_iy_i+y_ix_i)-\sum_{\gamma\in\Gamma\setminus\{1\}}c_\gamma\frac{\lambda_\gamma+1}{1-\lambda_\gamma}\sum_{i}\gamma_i\in \bfH_{t,k,c}^{\Gamma_n},\\
\frac 12\kappa_1&=t\sum_{i}(x_i y_i+n)\in C(V)^{\Gamma_n},\\
\Omega_{\Gamma_n,c,k}&=\sum_{\gamma\in\Gamma\setminus\{1\}}\frac{2c_\gamma}{1-\lambda_\gamma}\sum_i\gamma_i-\frac k2\sum_{i\neq j}\sum_\gamma s_{ij}\gamma_i\gamma_j^{-1} \in \bC[\Gamma_n]^{\Gamma_n}.
\end{aligned}
\end{equation}
The diagonal embedding $\Delta: \Gamma_n\to\bfH_{t,k,c}\otimes C(V)$ is defined by $\gamma_i\mapsto\gamma_i\otimes\tau_{\gamma_i}$, $s_{ij}\gamma_i\gamma_j^{-1}\mapsto s_{ij}\gamma_i\gamma_j^{-1}\otimes \tau_{s_{ij}\gamma_i\gamma_j^{-1}}$, where
\begin{equation}
\begin{aligned}
\tau_{\gamma_i}&=\frac{1-\lambda_\gamma}2 y_i x_i+1,\\
\tau_{s_{ij}\gamma_i\gamma_j^{-1}}&=\frac 12 (x_i-\lambda_\gamma x_j)(y_i-\lambda^{-1} y_j)+1.\\
\end{aligned}
\end{equation}

\section{Applications: unitarity, the Calogero-Moser space}

\subsection{A star operation} In \cite{ES}, Etingof and Stoica study unitary $\bfH_{t,c}$ modules with respect to the following star operation. Suppose that $t,c_\al\in \bR$ and let $\bar{ ~}$ denote the complex conjugation of $\fh$ and $\fh^*$ with respect to the real span of the coroots and roots, respectively. Let $\iota: V\to V$ denote the isomorphism induced by the symmetric bilinear form $\langle~,~\rangle$. More precisely,  set 
\[\iota(x)=\sum_i\langle x,y_i\rangle y_i,\ \iota(y)=\sum_i\langle y,x_i\rangle x_i, \quad x\in\fh^*,\ y\in\fh.\]
The star operation $\star$ is the anti-linear involutive anti-automorphism defined on generators by
\begin{equation}
w^\star=w^{-1},\quad v^\star=\overline{\iota(v)},\ \quad w\in W,\ v\in V.
\end{equation}
In particular, if we choose the basis $\{x_i\}$ of $\fh^*$ to lie in the real span of roots and let $\{y_i\}$ be the dual basis, we have
\begin{equation}
x_i^\star=y_i,\quad y_i^\star=x_i.
\end{equation}
We will assume implicitly from now on that the bases are of this form.

\

Define in $C(V)$ a star operation, denoted by $*$:
\begin{equation}
v^*=-\overline{\iota(v)}, \quad v\in V,
\end{equation}
which we extend to an anti-automorphism. With this definition, the spin module $S$ has a positive definite $*$-invariant hermitian form. Indeed, realizing $S$ as before on $\bigwedge\fh$, let $I=(i_1,\dots,i_k)$ be a set of indices written in increasing order and denote by $y_I:=y_{i_1}\wedge\dots\wedge y_{i_k}$ the corresponding basis element of $\bigwedge h.$ The form on $S$ is defined by
\begin{equation}
(y_I,y_J)_S:=\delta_{I,J},
\end{equation}
for any two multi-indices $I$ and $J$. One can check using the action on $S$ by the $x_i$, $y_i$'s that
\[(y_i\cdot y_I,y_J)_S=-(y_I,x_i\cdot y_J)_S,\]
which shows that the form $(~,~)_S$ is indeed $*$-invariant. This makes $S$ into a $*$-unitary $C(V)$-module. 

The star operations on $\bfH_{t,c}$ and $C(V)$ endow $\bfH_{t,c}\otimes C(V)$ with a star operation which we denote $\star$ too.  The following lemma is then clear.

\begin{lemma} The Dirac elements satisfy $\C D_x^\star=-\C D_y$ and $\C D_y^\star=-\C D_x.$ Therefore, $\C D$ is skew-adjoint with respect to $\star$, i.e., $$\C D^\star=-\C D.$$
\end{lemma}

\subsection{Unitary $\bfH_{t,c}$-modules}
Suppose that $M$ is a simple $\bfH_{t,c}$-module endowed with a $\star$-invariant hermitian form $(~,~)_M$. Then, we may define the $\star$-invariant product form on the $\bfH_{t,c}\otimes C(V)$-module $M\otimes S$:
\begin{equation}\label{e:prod-form}
(m_1\otimes s_1,m_2\otimes s_2)=(m_1,m_2)_M(s_1,s_2)_S,
\end{equation}
and it is extended sesquilinearly.

We have the following easy criterion.

\begin{proposition}\label{p:unit-criterion}
Suppose that $M$ is a  $\star$-unitary $\bfH_{t,c}$-module. Then:
\begin{equation}
(\C D^2a,a)\le 0, \text{ for all } a\in M\otimes S.
\end{equation} 
\end{proposition}

\begin{proof}
Since $M$ and $S$ are both unitary, so is $M\otimes S.$ Then 
\[0\le (\C Da,\C Da)=(\C D^\star \C D a,a)=-(\C D^2a,a).\]
\end{proof}

When $M$ is $\star$-unitary, the definition of Dirac cohomology becomes easier. Since $\C D$ is skew-adjoint, it follows that $\ker D_X\cap \im D_X=0$, and therefore $$H_D(X)=\ker D_X=\ker D_X^2.$$ Moreover, suppose $m\in\ker D_X$. Then $\C D_x m=-\C D_y m$ and 
\[(\C D_x m,\C D_x m)=(\C D_x m,-\C D_y m)=(\C D_x^2 m, m)=0.\]
If $M$ is unitary, it follows that $\C D_xm=0$. This argument shows that $\ker D_X=\ker \C D_x\cap \ker \C D_y$ if $M$ is unitary.

\subsection{Category $\CO$ for $\bfH_{1,c}$}

Set $t=1$. The categories $\CO$ for rational Cherednik algebras were introduced by Ginzburg, Guay, Opdam, and Rouquier in \cite{GGOR}, where their main properties are established as well.


The simple modules $V_{\delta,\sigma}$ for the semidirect product $S(\fh)\rtimes \bC[W]$ are labelled by  $\delta\in \fh^*$ and  the irreducible representations $(\sigma,V_\sigma)$ of the isotropy group $W_\delta$ of $\delta$ in $W$, and are obtained by Mackey induction. Furthermore, define the induced standard $\bfH_{1,c}$-module:
\begin{equation}
M(\delta,\sigma)=\bfH_{1,c}\otimes_{S(\fh)\rtimes \bC[W]} V_{\delta,\sigma}.
\end{equation}
It is clear that as a left $S(\fh^*)$-module, $M(\delta,\sigma)$ is isomorphic to $S(\fh^*)\otimes V_{\delta,\sigma}$.

\smallskip

Restrict to the case $\delta=0$, and denote the standard module by $M(\sigma).$ Let $L(\sigma)$ be its unique simple quotient.

\smallskip

Denote the $W$-character
\begin{equation}\label{e:epsilon}
\varepsilon:={\det}_{\fh}: W\to \bC^\times.
\end{equation}

We would like to understand the action of $\C D$ and $\C D^2$ on $M(\sigma)\otimes S$. From the formula for $\C D^2$ and the action of $\tau_s$ on $S$, it is apparent that the action of $\C D^2$ (Proposition \ref{p:square-dirac}) preserves the subspaces 
\[\C M_{k,\ell}(\sigma):= (S^k(\fh^*)\otimes V_{\sigma})\otimes {\bigwedge}^{\ell}\fh,\quad\text{for all } k\ge 0,\ 0\le\ell\le n.\]
Using formula (\ref{e:Omega-cherednik}), we see that for every $f\otimes v\in S^k(\fh^*)\otimes V_\sigma$:
\begin{equation}\label{h-action}
\Omega_{\bfH}(f\otimes v)=[\Omega_{\bfH},f]\otimes v+f\Omega_{\bfH}\otimes v=(2k+ n-N_c(\sigma)) f\otimes v,
\end{equation}
where $N_c(\sigma)$ is the scalar by which the central element $\Omega_{W,c}$ from (\ref{e:Omega-W-2}) acts on $V_\sigma$ under the representation $\sigma$. We also used here that $y_i$ acts by $0$ on $V_\sigma$.

By Lemma \ref{l:W-action}, the diagonal action of $W$ on $M(\sigma)\otimes S$ preserves degrees and on $\C M_{k,\ell}(\sigma)$ it is  the natural diagonal action of $W$ on 
\[S^k(\fh^*)\otimes V_\sigma \otimes {\bigwedge}^{\ell}\fh.\]

\begin{proposition}\label{p:scalar-action}
Let $\mu$ be an irreducible $W$-representation. If the $\mu$-isotypic component of $S^k(\fh^*)\otimes V_{\sigma}\otimes {\bigwedge}^{\ell}\fh$ (under the natural diagonal action of $W$) is nonzero, the square of the Dirac element $\C D^2$ acts on it by the scalar
\begin{equation}\label{D2-scalar}
-2(k+n-\ell)+N_c(\sigma)-N_c(\mu).
\end{equation}
\end{proposition}

\begin{proof}
This is immediate from the formula for $\C D^2$ (Proposition \ref{p:square-dirac}) since the action of $\Omega_\bfH\otimes 1$ is given by the scalar (\ref{h-action}) and the action of $1\otimes\frac 12\kappa_1$ is by $-(n-2\ell).$
\end{proof}

The classical construction of contravariant forms on Verma modules of a semisimple Lie algebra, can be adapted to this setting to show that $M(\sigma)$ admits a hermitian invariant form (see \cite{ES}). 

\begin{corollary} Let $\sigma$ be an irreducible $W$-module.
\begin{enumerate}
\item Suppose that $M(\sigma)$ is irreducible. Then $M(\sigma)$ is $\star$-unitary only if
\begin{equation}
N_c(\sigma)-N_c(\mu)\le 2(k+\ell), 
\end{equation}
for  all irreducible $W$-representations $\mu$ that appear in the natural diagonal action of $W$ on $S^k(\fh^*)\otimes V_\sigma\otimes {\bigwedge}^\ell\fh,$ $k\ge 0$, $0\le \ell\le n$.
\item The simple module $L(\sigma)$ is unitary only if
\begin{equation}\label{L-unit-ineq}
N_c(\sigma)-N_c(\mu)\le 2\ell, 
\end{equation}
for  all irreducible $W$-representations $\mu$ that appear in the natural diagonal action of $W$ on $\sigma\otimes {\bigwedge}^\ell\fh,$  $0\le \ell\le n$.
\end{enumerate}
\end{corollary}

\begin{proof}
Apply the unitary criterion in Proposition \ref{p:unit-criterion} and use formula (\ref{D2-scalar}) for the action of $\C D^2$. For (2), use that $\sigma$ occurs in the restriction of $L(\sigma)$ to $W$.
\end{proof}

\begin{remark}
In the case when $\ell=1$ and the parameter function $c$ is constant, formula (\ref{L-unit-ineq}) recovers \cite[Corollary 3.6]{ES}.
\end{remark}

\subsection{Dirac cohomology in category $\CO$}  Every module in category $\CO$ is $\Omega_\bfH$-admissible  in the sense of Definition \ref{d:coh}, and so the notion of Dirac cohomology makes sense. If $f\otimes v$ is a simple tensor element of $M(\sigma)$ where $f\in S(\fh^*)$ and $v\in V_\sigma$ and $p\in S$, then the Dirac operator acts as
\[D_{M(\sigma)}(f\otimes v\otimes p)=\sum_i x_if\otimes v\otimes y_ip+\sum_i y_if\otimes v\otimes x_i p.\]
If we take $f=1$ and $p^{(n)}\in {\bigwedge}^n\fh$, then it is immediate that
$D_{M(\sigma)}(1\otimes v\otimes p^{(n)})=0$. This implies that 
\begin{equation}
1\otimes \sigma\otimes {\bigwedge}^n\fh\subset \ker D_{M(\sigma)}.
\end{equation}

\smallskip

Conversely, suppose $\wti f$ is an element in the isotypic component of an irreducible $W$-representation $\mu$ in $\ker D_X$ coming from the natural $W$-action on an $S^k(\fh^*)\otimes V_\sigma\otimes{\bigwedge}^\ell\fh^*$ . Then applying Proposition \ref{p:scalar-action} we see that
\begin{equation}\label{e:sigma-mu}
N_c(\sigma)-N_c(\mu)=2(k+n-\ell).
\end{equation}
In particular, suppose that $\mu=\sigma$. Then this equation implies that $k=\ell-n$, which can only be satisfied if $k=0$ and $\ell=n$. This means that the only copy of $\sigma\otimes\varepsilon$ in $\ker D_{M(\sigma)}$ (as a $\tau(W)$-representation) is $1\otimes \sigma\otimes {\bigwedge}^n\fh.$

\begin{proposition}\label{p:coh-verma} For every $\sigma\in \Irr(W)$, we have
\[\dim\Hom_{W}[\sigma\otimes\varepsilon, H_D(M(\sigma))]=1.
\]
In particular, $H_D(M(\sigma))\neq 0.$
\end{proposition}

\begin{proof}
The previous discussion shows that $\dim\Hom_{W}[\sigma\otimes\varepsilon, \ker D_{M(\sigma)}]=1$, so we only need to show that $1\otimes V\otimes {\bigwedge}^n\fh\not\subset \im D_{M(\sigma)}.$ From the definition of $D_{M(\sigma)}$, we see that
\begin{equation}
D_{M(\sigma)}(S^k(\fh^*)\otimes V_\sigma\otimes{\bigwedge}^\ell\fh\subseteq \left( S^{k+1}(\fh^*)\otimes V_\sigma\otimes{\bigwedge}^{\ell+1}\fh\right)+\left(S^{k-1}(\fh^*)\otimes V_\sigma\otimes{\bigwedge}^{\ell-1}\fh\right),
\end{equation}
where $S^{-1}(\fh^*)=0$ by convention. The claim follows.
\end{proof}

\subsection{Vogan's Dirac morphism: case $t=1$} When $t\neq 0$, by scaling the parameters $c$ appropriately, it is sufficient to consider the case $t=1$. As in (\ref{e:alg-A}), define $\bfA_{1,c}=Z_{\bfH_{1,c}\otimes C(V)}(\wti \Omega_{\bfH_{1,c}}).$ From Remark \ref{r:commute-Omega}, we know that for every $x\in \fh^*$ and every $y\in \fh$, we have
\[ x\otimes y, y\otimes x\in \bfA_{1,c}.\]
Define $B_c$ to be the subalgebra of $\bfH_{1,c}\otimes C(V)$ generated by:
\begin{equation}
\fh\otimes\fh^*,\ \fh^*\otimes\fh,\ \Delta(\bC[W]).
\end{equation}
In particular, $\C D\in B_c^W$, and so $\wti\Omega_{\bfH}\in B_c^W.$ Apply (\ref{e:spec-morphism}) for $\C B:=Z_{+}(B_c)\subset B^W$, where $Z_{+}(B_c)$ denotes the even part of the center of $B_c$, i.e.,
\[Z_{+}(B_c)=\{z\in B_c: \ep(z)=z,\ [z,b]=0,\text{ for all }b\in B_c\}.\]
Clearly, $\wti\Omega_{\bfH_{1,c}}\in Z_+(B_c).$ Then, the general results about the Vogan's Dirac morphism $\zeta$ give a canonical morphism:
\begin{equation}
\zeta^*_{1,c}: \Irr(W)\to \Spec Z_{+}(B_c).
\end{equation}

\begin{example}
Suppose the root system is of type $A_1$, $\fh=\< y\>$, $\fh^*=\< x\>$, $W=\{1,s\}$ and $\langle y,x\rangle=1$. Denote $\wti x=x\otimes y$, $\wti y=y\otimes x$, and $\wti\Omega:=\wti\Omega_{1,c}$ and identify $\Delta(s)$ with $s$. Then $B_c$ is generated by $\wti x$, $\wti y$, $s$, and $\wti\Omega$ subject to the relations
\begin{equation}
\begin{aligned}
&s^2=1,\ s\wti x s=\wti x,\ s\wti y s=\wti y;\\
&[\wti\Omega,\wti x]=[\wti\Omega,\wti y]=[\wti\Omega,s]=0;\\
&\wti x^2=0,\ \wti y^2=0;\\
&\wti x\wti y+\wti y \wti x=-\wti\Omega-c s.
\end{aligned}
\end{equation}
In this case, $Z_+(B_c)=Z(B_c)=\bC\<\wti\Omega,s\>$, so we may identify $\Spec Z(B_c)=\bC\times \bZ/2\bZ$. If $\lambda\in \bC$ and $\sigma\in\Irr(W)$, consider the finite algebra $B_{c,\lambda,\sigma}=B_c/\<\wti\Omega-\lambda,s-\sigma(s)\>.$ (Of course, $\sigma(s)\in\{\pm 1\}$.) This is a $4$-dimensional algebra generated by $\wti x$, $\wti y$ subject to relations
\begin{equation}
\wti x^2=\wti y^2=0,\quad \wti x\wti y+\wti y\wti x=-\lambda-c\sigma(s).
\end{equation}
Thus, $B_{c,\lambda,\sigma}$ is isomorphic to the Clifford algebra for the two dimensional space $\bC\<x,y\>$ for the symmetric bilinear form with matrix $\frac 12\left(\begin{matrix}0 &\lambda+c\sigma(s)\\\lambda+c\sigma(s)&0\end{matrix}\right)$.
The morphism $\zeta_{1,c}^*$ is in this case
\[
\zeta_{1,c}^*: \bZ/2\bZ\to \bC\times \bZ/2\bZ,\quad \sigma\mapsto (-c\sigma, \sigma).
\]
\end{example}

\subsection{Vogan's Dirac morphism: case $t=0$}\label{s:CM} When $t=0$, the algebra $\bfH_{0,c}$ has a large center. By \cite{EG}, the center $Z(\bfH_{0,c})$ contains the subalgebra $\fk m:=S(\fh)^W\otimes S(\fh^*)^W$ and it is a free $\fk m$-module of rank $|W|$. Notice that $\wti\Omega_{\bfH_{0,c}}=\Omega_{\bfH_{0,c}}\otimes 1$ is in fact in $Z(\bfH_{0,c})\otimes 1$, but not in $\fk m\otimes 1$. 

In this case, we have $\bfA_{0,c}=\bfH_{0,c}\otimes C(V)$, so we may choose $\C B=Z(\bfH_{0,c})\otimes 1.$ Thus, we obtain an algebra homomorphism
\begin{equation}\label{e:hom-0}
\zeta_{0,c}: Z(\bfH_{0,c})\to \bC[W]^W,
\end{equation}
and the dual morphism
\begin{equation}\label{e:spec-0}
\zeta_{0,c}^*: \Irr(W)\to \Spec Z(\bfH_{0,c})=X_c(W).
\end{equation}
Here, $X_c(W)$ is the Calogero-Moser space \cite{EG}. The inclusion $\fk m\subset Z(\bfH_{0,c})$ induces a surjective morphism
\begin{equation}\label{e:Y}
\Upsilon: X_c(W)\to \fh^*/W\times \fh/W.
\end{equation}
Let $\fk m_+$ be the augmentation ideal of $\fk m$ and define similarly $S(\fh)^W_+$ and $S(\fh^*)^W_+$.
\begin{theorem}\label{t:CM}
The algebra homomorphism (\ref{e:hom-0}) factors through $Z(\bfH_{0,c})/\fk m_+$:
\[\zeta_{0,c}: Z(\bfH_{0,c})/\fk m_+\to \bC[W]^W,
\]
 and so the dual morphism is
 \begin{equation}
\zeta_{0,c}^*:\Irr(W)\to \Upsilon^{-1}(0).
\end{equation}
\end{theorem}

\begin{proof}
In light of the definition of $\zeta_{0,c}$ and Theorem \ref{t:ker-dW}, we need to prove that 
\[S(\fh)^W_+\otimes 1\subset \im d_\triv\text{ and }S(\fh^*)^W_+\otimes 1\subset \im d_\triv,\]
where $d_\triv: (\bfH_{0,c}\otimes C(V))^W\to (\bfH_{0,c}\otimes C(V))^W,$ $d(a)=\C D a-\ep(a)\C D.$ Notice that in this case, $\det_V(\Delta(W))=1$, hence the absence of the $\det_V$ in the target of the map. 

Embed ${\bigwedge\fh}\subset C(V)$ in the natural way. We show first that the argument in section \ref{s:induction} implies
\begin{equation}\label{e:step}
(S(\fh)\otimes {\bigwedge}\fh)^W\cap \ker d_\triv \subset \im d_\triv+(1\otimes ({\bigwedge \fh})^W)\cap \ker d_\triv=\im d_\triv\oplus \bC(1\otimes 1).
\end{equation}
(The last equality follows from Theorem \ref{t:ker-dW}.)
 Let $a\in (S(\fh)\otimes \bigwedge\fh)^W$ have degree $n\ge 1$ and $d_\triv(a)=0$. Then taking graded objects, we have $\bar d'(\bar a)=\bar d_\triv(\bar a)=0$, where $\bar a\in  (S(\fh)\otimes \bigwedge\fh)^W$ regarded in $\bfH_{0,0}\otimes C(V)$ has degree $n$.  Her $\bar d'$ is the Koszul differential from (\ref{e:koszul}). Thus by (\ref{e:koszul}), $\bar a\in\im \bar d'$, i.e., there exists $\bar b\in  (S(V)\otimes C(V))^W$ of degree $n-1$ such that $\bar a=\bar d(\bar b)$. Moreover, notice that in $\bfH_{0,0}\otimes C(V)$ if $\bar f\in S(\fh)$ and $\bar y\in\fh$, then
\begin{equation}
\begin{aligned}
\bar d(\bar f\otimes y)&=\sum_i \bar x_i\bar f\otimes (y_iy+yy_i)+\sum_i\bar f\bar y_i\otimes (x_iy+yx_i)=\sum_i \bar f\bar y_i\otimes (-2\langle x_i,y\rangle)\\
&=-2\bar f\bar y\otimes 1.
\end{aligned}
\end{equation}
Together with the derivation property of $\bar d$, this shows that $\bar b$ can in fact be chosen in $(S(\fh)\otimes \bigwedge\fh)^W$. Proceeding as in section \ref{s:induction}, choose $b=\bar b$ but regarded in  $(\bfH_{0,c}\otimes C(V))^W$ (this is well-defined since $S(\fh)$ is abelian in $\bfH_{0,c}$ too) and then $a-d_\triv(b) \in (S(\fh)\otimes \bigwedge \fh)^W$ of degree at most $n-1$. Also, $d_\triv(a-d_\triv(b))=0.$ The inclusion (\ref{e:step}) follows by induction. 

Finally, from the definition of $d$, it is clear that $\im d\subset (\bC[W]\otimes S(V)_+)\otimes C(V)$, and therefore (\ref{e:step}) implies that $S(\fh)^W_+\otimes 1\subset d_\triv$.

The case of $\fh^*$ is entirely similar.
\end{proof}

\subsection{Finite dimensional modules for $\bfH_{0,c}$} Retain the notation from section \ref{s:CM}. Following \cite{Go}, let us consider the  "baby Verma modules" for $\bfH_{0,c}.$ Define
\[\bar\bfH_{0,c}=\bfH_{0,c}/\fk m_+\bfH_{0,c}\]
This is a finite dimensional algebra of dimension $|W|^3$, isomorphic to $S(\fh^*)_W\otimes S(\fh)_W\otimes \bC[W]$ as a vector space, where we denote by $S(\fh)_W=S(\fh)/S(\fh)^W_+$ the graded algebra of coinvariants and similarly for $S(\fh^*)_W$. 

For every $(\sigma,V_\sigma)\in \Irr(W)$, let
\begin{equation}
\bar M(\sigma)=\bar\bfH_{0,c}\otimes_{S(\fh)_W\rtimes \bC[W]}V_\sigma
\end{equation}
be the baby Verma module induced from $\sigma$. Here $S(\fh)_W$ acts by $0$ on $V_\sigma.$ 

\begin{theorem}[{\cite[Proposition 4.3]{Go}}]
\begin{enumerate}
\item For every $\sigma\in\Irr(W)$, the module $\bar M(\sigma)$ is indecomposable and it has a unique simple quotient $\bar L(\sigma)$.
\item The set $\{\bar L(\sigma):\sigma\in\Irr(W)\}$ gives a complete list of non isomorphic simple $\bar\bfH_{0,c}$-modules.
\end{enumerate}
\end{theorem}

As a consequence (\cite[\S5.4]{Go}), the map 
\begin{equation}\label{e:Theta}
\Theta:\Irr(W)\to \Upsilon^{-1}(0)=\Spec Z(\bar\bfH_{0,c}),
\end{equation}
 given by mapping $\bar M(\sigma)$ to its central character, is surjective.

\smallskip

Consider $\bar D$ the Dirac operator for $\bar H_{0,c}$. The same argument as for Proposition \ref{p:coh-verma} shows that 
\begin{equation}
\dim\Hom_{W}[\sigma\otimes\varepsilon, H_D(\bar M(\sigma))]\neq 0.
\end{equation}
Applying Theorem \ref{t:vogan-conj} (and the remark following it) we are led to the following corollary.

\begin{corollary}\label{c:zeta=theta} The morphism $\zeta^*_{0,c}:\Irr(W)\to \Upsilon^{-1}(0)$ from Theorem \ref{t:CM} is the $\varepsilon$-dual of the morphism $\Theta$ from (\ref{e:Theta}), i.e., $\Theta(\sigma)=\zeta^*_{0,c}(\sigma\otimes\varepsilon)$.
\end{corollary}

The fibers of the map $\Theta$ from (\ref{e:Theta}) are expected to be related to the partition of $\Irr(W)$ into families of representations, in the sense of Lusztig \cite{Lu2} for real reflection groups, and \cite{Ro} for complex reflection groups. More precisely, it has been conjectured by Gordon and Martino \cite{GM} that the partition of $\Irr(W)$ according to the fibers of $\Theta$ refines the partition into Lusztig-Rouquier families, and that, at least for finite reflection groups, the two partitions coincide. This is known to hold in many cases, for example in type $A$ \cite{EG}, and for the class of complex reflection groups $G(m,d,n)$ \cite{GM} and \cite{Be}. 

 We hope that the point of view offered by Theorem \ref{t:CM} and Corollary \ref{c:zeta=theta}, as well as the related methods of the Dirac operator will help in understanding the relation with families of representations and cells in finite reflection groups. 

\begin{example}
Let $W$ be the Weyl group of type $B_2$ and let $\sigma=(11\times 0)$ be the one dimensional $W$-representation where the long reflections act by $-1$ and the short reflections by $1$. Assume that the parameter function $c$ is constant. One can easily  verify that for every $x\in\fh^*$, $y\in\fh$, the commutator $[y,x]\in\bC[W]$ acts by $0$ on $\sigma.$ This means that $\sigma$ can be extended to a simple module $\bar L(\sigma)$ of $\bfH_{0,c}$ by letting both $x$ and $y$ act by $0$. The Dirac operator then acts identically by $0$, and thus the Dirac cohomology equals
\begin{equation}
H_D(\bar L(\sigma))=\sigma\otimes {\bigwedge}\fh=(11\times 0)+(1\times 1)+(0\times 2),
\end{equation}
where $1\times 1$ is the reflection representation and $0\times 2$ is the sign twist of $11\times 0.$ Theorem \ref{t:vogan-conj} implies then that $\zeta^*_{0,c}(11\times 0)=\zeta^*_{0,c}(1\times 1)=\zeta^*_{0,c}(0\times 2)$, and therefore, by Corollary \ref{c:zeta=theta}, $\{ 11\times 0,1\times 1,0\times 2\}$ is contained in the same fiber of $\Theta.$ Of course, $\{ 11\times 0,1\times 1,0\times 2\}$ is the first nontrivial example of a Lusztig family of $W$-representations.
\end{example}

\ifx\undefined\bysame
\newcommand{\bysame}{\leavevmode\hbox to3em{\hrulefill}\,}
\fi

\end{document}